\documentclass[a4paper,12pt]{amsart}

\usepackage{amsmath}
\usepackage{amssymb}
\usepackage{amsfonts}

\newtheorem{thm}{Theorem}[section]
\newtheorem{prop}[thm]{Proposition}
\newtheorem{lem}[thm]{Lemma}
\newtheorem{cor}[thm]{Corollary}
\newtheorem{rem}[thm]{Remark}
\newtheorem{rems}[thm]{Remarks}
\newtheorem{defi}[thm]{Definition}
\newtheorem{exo}{\bf\large Exercice}

\newcommand{\R}{\mathbb{R}}

\newcommand{\N}{\mathbb{N}}
\newcommand{\C}{\mathbb{C}}

\newcommand{\I}{\infty}

\newcommand{\Sum}{\displaystyle \sum}

\newcommand{\Int}{\displaystyle \int}

\newcommand{\Inf}{\displaystyle \inf}
\newcommand{\Sup}{\displaystyle \sup}

\newcommand{\Liminf}{\displaystyle \liminf}
\newcommand{\Limsup}{\displaystyle \limsup}

\newcommand{\beq}{\begin{eqnarray}}
\newcommand{\eeq}{\end{eqnarray}}
\newcommand{\bq}{\begin{equation}}
\newcommand{\eq}{\end{equation}}
\newcommand{\beqn}{\begin{eqnarray*}}
\newcommand{\eeqn}{\end{eqnarray*}}
\newcommand{\bex}{\begin{exo}}
\newcommand{\eex}{\end{exo}}
\newcommand{\ben}{\begin{enumerate}}
\newcommand{\een}{\end{enumerate}}

\let\de=\delta

\let\r=\rho

\let\vf=\varphi

\def\cC{{\mathcal C}}
\def\cD{{\mathcal D}}

\def\cL{{\mathcal L}}

\def\cP{{\mathcal P}}

\def\na{\nabla}

\def\eqdefa{\buildrel\hbox{\footnotesize def}\over =}

\author{Hajer Bahouri}
\address{University of Tunis ElManar, Faculty of Sciences of Tunis,
Department of Mathematics.} \email{hajer.bahouri@fst.rnu.tn}
\thanks{H.B. is grateful to the Laboratory of
PDE and Applications at the Faculty of Sciences of Tunis.}

\author{Mohamed Majdoub}
\address{University of Tunis ElManar, Faculty of Sciences of Tunis,
Department of Mathematics.} \email{mohamed.majdoub@fst.rnu.tn}
\thanks{M. M. is grateful to the Laboratory of
PDE and Applications at the Faculty of Sciences of Tunis.}

\author{Nader Masmoudi}
\address{New York University \\
The Courant Institute for Mathematical Sciences.}
\email{masmoudi@courant.nyu.edu}
\thanks{N. M is partially supported by an NSF Grant DMS-0703145}

\title[On the lack of compactness...]
{On  the lack of compactness in the 2D critical Sobolev embedding}
\date{\today}

\begin{document}

\begin{abstract}
This paper is devoted to the description of the lack of
compactness of $H^1_{rad}(\R^2)$ in the Orlicz space. Our result
is expressed in terms of the concentration-type examples derived
by P. -L.  Lions in \cite{Lions1}. The approach that we adopt to
establish this characterization is completely different from the
methods used in the study of the lack of compactness of Sobolev
embedding in Lebesgue spaces and take into account the variational
aspect of Orlicz spaces. We also investigate the feature of the
solutions of non linear wave equation with exponential growth, where the Orlicz norm plays a decisive role.
\end{abstract}


\subjclass[2000]{...}



\maketitle \tableofcontents


\section{Introduction}


\subsection{Lack of compactness in the Sobolev embedding in Lebesgue spaces}


Due to the scaling invariance of the critical Sobolev embedding
$$\dot H^s(\R^d)\longrightarrow L^{p}(\R^d),$$ in the case where $d \geq 3$ with $0\leq
s<d/2$ and $p=2d/(d-2s),$ no compactness properties may be expected.
Indeed if~$u \in \dot H^s \setminus\{0\}$, then for any
sequence~$(y_n)$ of points of~$\R^d $ tending to the infinity and
for any sequence~$(h_n)$ of positive real numbers tending to~$0 $ or
to   infinity, the sequences~$(\tau_{y_n} u )$ and~$(\delta_{h_n} u
)$, where we denote~$\delta_{h_n} u(\cdot)= \frac{1}{h_n^{\frac d
p}}u(\frac{\cdot}{h_n})$, converge weakly to~$0$ in~$\dot H^s $ but
are not relatively compact in~$ L^p $ since~$\| \tau_{y_n} u
\|_{L^p} = \|
 u \|_{L^p} $  and ~$\| \delta_{h_n} u \|_{L^p} =  \|
 u \|_{L^p} $. \\

\vskip 2ex

After the pioneering works of P. -L.  Lions \cite{Lions1} and
\cite{Lions2}, several works have been devoted to the study of the
lack of compactness in critical Sobolev embeddings, for the sake of
geometric problems and the understanding of   features of solutions
of nonlinear partial differential equations. This question was
investigated through several angles: for instance, in \cite{Ge1} the
lack of compactness is describe in terms of defect measures, in
\cite{Ge2} by means of profiles  and in \cite{jaffard} by the use of
nonlinear wavelet approximation theory. Nevertheless, it has been
shown in all these results that translational and scaling invariance
are the sole responsible for the defect of compactness of the
embedding of~$\dot H^s$ into~$
L^p $ and more generally in Sobolev spaces in the~$ L^q$ frame. \\

As it is pointed above, the study of  the lack of compactness in
critical Sobolev embedding supply us numerous information on
solutions of nonlinear  partial differential equations whether in
elliptic frame or evolution frame. For example, one can mention the
description of bounded energy sequences of solutions to the
defocusing semi-linear wave equation $$ \square u + u^5 =0$$ in $ \R
\times \R^3$, up to remainder terms small in energy norm in
\cite{BG} or the sharp estimate of the span time life of the
focusing critical semi-linear  wave equation by means of the size of
energy of the Cauchy data in the remarkable work of \cite{km}.
\\

Roughly speaking, the lack of compactness in the critical Sobolev
embedding$$\dot H^s(\R^d)\hookrightarrow L^{p}(\R^d)$$ in the case
where $d \geq 3$ with $0\leq s<d/2$ and $p=2d/(d-2s),$ is
characterized in the following terms: a  sequence~$(u_n)_{n\in \N}$
bounded in~$\dot H^s(\R^d)$ can be decomposed up to a subsequence
extraction on
 a finite sum of orthogonal profiles such that the remainder
converges to zero in~$L^{p}(\R^d)$ as the number of the sum and
~$n$ tend to~$\infty$.\\
This description still holds in the more general case of Sobolev
spaces in the~$ L^q$ frame (see \cite{jaffard}).


\subsection{Critical $2D$ Sobolev embedding}


It is well known that $H^1(\R^2)$ is continuously embedded  in all
lebesgue spaces $L^p(\R^2)$ for $2\leq p<\infty$, but not in
$L^{\infty}(\R^2 )$.  A  short proof of this fact is given in
Appendix A for the convenience of the reader. On the other hand, it
is also known (see for instance  \cite{km}) that  $H^1(\R^2)$ embed
in  ${\rm BMO} (\R^2) \cap L^2 (\R^2)    $,  where ~$BMO(\R^d )$
denotes the space of bounded mean oscillations which is  the space
of locally integrable functions~$f$ such that \[ \|f\|_{BMO}\eqdefa
\sup_{B}\frac 1 {|B|}\int_B |f-f_B|\,dx< \infty \quad
\mbox{with}\quad f_B\eqdefa \frac 1 {|B|}\int_B f \,dx.\] The above
supremum being taken over
the set of Euclidean balls~$B$,~$|\cdot |$ denoting the Lebesgue measure.  \\

In this paper, we rather investigate the lack of compactness in
Orlicz space ${\mathcal L}$ (see Definition \ref{deforl} below)
which arises naturally in the study of non linear wave equation with
exponential growth. As, it will be shown in  Appendix B, the spaces
 ${\mathcal L}$ and {\rm BMO} are not comparable.  \\

Let us now introduce the so-called Orlicz spaces on $\R^d$ and some
related basic facts. (For the sake of completeness, we postpone to
an appendix some additional properties on Orlicz spaces).

\begin{defi}\label{deforl}\quad\\
Let $\phi : \R^+\to\R^+$ be a convex increasing function such that
$$
\phi(0)=0=\lim_{s\to 0^+}\,\phi(s),\quad
\lim_{s\to\infty}\,\phi(s)=\infty.
$$
We say that a measurable function $u : \R^d\to\C$ belongs to
$L^\phi$ if there exists $\lambda>0$ such that
$$
\Int_{\R^d}\,\phi\left(\frac{|u(x)|}{\lambda}\right)\,dx<\infty.
$$
We denote then \bq \label{norm}
\|u\|_{L^\phi}=\Inf\,\left\{\,\lambda>0,\quad\Int_{\R^d}\,\phi\left(\frac{|u(x)|}{\lambda}\right)\,dx\leq
1\,\right\}. \eq
\end{defi}
It is easy to check that $L^\phi$ is a $\C$-vectorial space and
$\|\cdot\|_{L^\phi}$ is a norm. Moreover, we have the following
properties.\\ \noindent$\bullet$ For $\phi(s)=s^p,\, 1\leq
p<\infty$,  $L^\phi$ is nothing else than the  Lebesgue space
$L^p$.\\ \noindent$\bullet$ For  $\phi_\alpha(s)={\rm e}^{\alpha
s^2}-1$, with $\alpha>0$,  we claim  that
$L^{\phi_\alpha}=L^{\phi_1}.$ It is actually a direct consequence
of Definition \ref{deforl}.\\ \noindent$\bullet$ We may replace in
\eqref{norm} the number $1$ by any positive constant. This change
the norm $\|\cdot\|_{L^\phi}$ to an equivalent norm.\\
\noindent$\bullet$  For $u\in L^\phi$ with $A:=\|u\|_{L^\phi}>0$,
we have the following property \bq \label{norm1}
\left\{\,\lambda>0,\quad\Int_{\R^d}\,\phi\left(\frac{|u(x)|}{\lambda}\right)\,dx\leq
1\,\right\}=[A, \infty[\,. \eq

In what follows we shall fix  $d=2$, $\phi(s)={\rm e}^{ s^2}-1$ and
denote the Orlicz space $L^\phi$ by ${\mathcal L}$ endowed with the
norm $\|\cdot\|_{\mathcal L}$ where the number $1$ is replaced by
the constant $\kappa$ that will be fixed in Identity  \eqref{Mos2}
below.  As it is already mentioned, this change does not have any
impact on  the definition of Orlicz space. It is easy to see that
${\mathcal L}\hookrightarrow L^p$
for every $2\leq p<\infty$. \\
The 2D critical Sobolev embedding in Orlicz space ${\mathcal L}$
states as follows:
\begin{prop}\label{prop2D-embed}
\begin{equation}
\label{2D-embed} \|u\|_{{\mathcal
L}}\leq\frac{1}{\sqrt{4\pi}}\|u\|_{H^1}.
\end{equation}
\end{prop}
\begin{rems}\quad\\
{\bf a)} Inequality \eqref{2D-embed} is insensitive to space translation but not invariant under scaling nor oscillations.\\
{\bf b)} The embedding of $H^1(\R^2)$ in  ${\mathcal L}$  is sharp
within the context of Orlicz spaces. In other words, the target
space $\cL$ cannot be replaced by an essentially smaller Orlicz
space. However, this target space can be improved if we allow
different function spaces than Orlicz spaces. More precisely \bq
\label{sharp-embed} H^1(\R^2)\hookrightarrow BW(\R^2), \eq where the
Br{\'e}zis-Wainger space $BW(\R^2)$ is defined via
$$
\|u\|_{BW}:=\Big(\int_0^1\Big(\frac{u^*(t)}{\log({\rm
e}/t)}\Big)^2\,\frac{dt}{t}\Big)^{1/2}+\Big(\int_1^\I\,u^*(t)^2\,dt\Big)^{1/2},
$$
where $u^*$ denotes the rearrangement function of $u$ given by
$$
u^*(t)=\inf\Big\{\,\lambda>0;\quad |\{x;\quad |u(x)|>\lambda\}|\leq
t\,\Big\}.
$$
The embedding \eqref{sharp-embed} is sharper than \eqref{2D-embed} as $BW(\R^2)\varsubsetneqq \cL$. It is also optimal with respect to all rearrangement invariant Banach function spaces. For more details on this subject, we refer the reader to \cite{BW, Cianchi, EKP, Hans, HMT, MP-PAMS}.\\
{\bf c)} In higher dimensions ($d=3$ for example), the equivalent of Embedding \eqref{sharp-embed} is
$$
H^1(\R^3)\hookrightarrow L^{6,2}(\R^3),
$$
where $L^{6,2}$ is the classical Lorentz space. Notice that
$L^{6,2}$ is a rearrangement invariant Banach space but not an
Orlicz space.
\end{rems}
To end this short introduction to Orlicz spaces, let us point out
that the embedding~\eqref{2D-embed} derives immediately from the
following Trudinger-Moser type inequalities:
\begin{prop} \label{mostrud} Let $\alpha\in [0,4\pi [$. A constant $c_\alpha$ exists
such that
\begin{equation}
\label{Mos1} \Int_{\R^2}\,\left({\rm e}^{\alpha
|u|^2}-1\right)\,dx\leq c_\alpha \|u\|_{L^2(\R^2)}^2
\end{equation}
for all $u$ in $H^1(\R^2)$ such that $\|\nabla
u\|_{L^2(\R^2)}\leq1$. Moreover, if $\alpha\geq 4\pi$, then
(\ref{Mos1}) is false.
\end{prop}
A first proof of these inequalities using rearrangement can be
found in \cite{AT}. In other respects, it is well known (see for
instance \cite{Ruf}) that the value $\alpha=4\pi$ becomes
admissible in (\ref{Mos1}) if we require $\|u\|_{H^1(\R^2)}\leq1$
rather than $\|\nabla u\|_{L^2(\R^2)}\leq1$. In other words, we
have
\begin{prop}\label{proptm}
\begin{equation}
\label{Mos2} \sup_{\|u\|_{H^1}\leq 1}\;\;\Int_{\R^2}\,\left({\rm
e}^{4\pi |u|^2}-1\right)\,dx:=\kappa<\infty,
\end{equation} and this is false for $\alpha>4\pi$.
\end{prop}
Now, it is obvious that Estimate (\ref{Mos2}) allows to prove
Proposition~\ref{prop2D-embed}. Indeed, without loss of generality,
we may assume that $\|u\|_{H^1}=1$  which leads under
Proposition~\ref{proptm}  to the inequality
 $\|u\|_{{\mathcal L}}\leq\frac{1}{\sqrt{4\pi}}$, which is the desired result.

\begin{rem}
Let us mention that a sharp form of Trudinger-Moser inequality in bounded domain was obtained in \cite{AD}.
\end{rem}

\subsection{Lack of compactness in 2D critical Sobolev embedding in Orlicz space}


The embedding~$H^1\hookrightarrow{\mathcal L}$  is non compact at
least for two reasons. The first reason is the lack of compactness
at infinity. A typical example is $u_k(x)=\varphi(x+x_k)$ where
$0\neq\varphi\in{\mathcal D}$ and $|x_k|\to\infty$. The second
reason is of concentration-type derived by  P.-L.  Lions
\cite{Lions1, Lions2} and illustrated by the following  fundamental
example  $f_{\alpha}$ defined by:
\begin{eqnarray*}
 f_{\alpha}(x)&=&\; \left\{
\begin{array}{cllll}0 \quad&\mbox{if}&\quad
|x|\geq 1,\\\\ -\frac{\log|x|}{\sqrt{2\alpha\pi}} \quad
&\mbox{if}&\quad {\rm e}^{-\alpha}\leq |x|\leq 1 ,\\\\
\sqrt{\frac{\alpha}{2\pi}}\quad&\mbox{if}&\quad |x|\leq {\rm
e}^{-\alpha},
\end{array}
\right.
\end{eqnarray*}
where $\alpha>0$. \\

Straightforward  computations show that
$\|f_{\alpha}\|_{L^2(\R^2)}^2=\frac{1}{4\alpha}(1-{\rm
e}^{-2\alpha})-\frac{1}{2}{\rm e}^{-2\alpha}$ and $\|\nabla
f_{\alpha}\|_{L^2(\R^2)}=1$. Moreover, it can be seen easily that
$f_{\alpha}\rightharpoonup 0$ in $H^1(\R^2)$ as $\alpha\to \infty$
or $\alpha\to 0$. However, the lack of compactness of this sequence
in the Orlicz space ${\cL}$ occurs only when $\alpha$ goes to
infinity. More precisely, we have
\begin{prop}
\label{f_alpha}
~$f_{\alpha}$ denoting the sequence defined above, we have the following convergence results:\\
\noindent{\bf a)} $\|f_{\alpha}\|_{\mathcal L}\rightarrow \frac{1}{\sqrt{4\pi}}\quad\mbox{as}\quad \alpha\to\infty\,.$\\
\noindent{\bf b)} $\|f_{\alpha}\|_{\mathcal L}\rightarrow
0\quad\mbox{as}\quad \alpha\to0\,.$
\end{prop}
\begin{proof}[Proof of Proposition \ref{f_alpha}] Let us first go to the proof of the first assertion.
If
$$
\int\,\left({\rm
e}^{\frac{|f_\alpha(x)|^2}{\lambda^2}}-1\right)\,dx\leq \kappa,
$$
then
$$
2\pi\,\int_0^{{\rm e}^{-\alpha}}\,\left({\rm e}^{\frac{\alpha}{2\pi
\lambda^2}}-1\right)\,r\,dr\leq \kappa,
$$
which implies that
$$
\lambda^2\geq\frac{\alpha}{2\pi\log\left(1+\frac{\kappa{\rm
e}^{2\alpha}}{\pi}\right)}.
$$
It follows that
$$
\liminf_{\alpha\to\infty}\,\|f_\alpha\|_{\mathcal
L}\geq\frac{1}{\sqrt{4\pi}}\,.
$$
To conclude, it suffices to show that
$$
\limsup_{\alpha\to\infty}\,\|f_\alpha\|_{\mathcal
L}\leq\frac{1}{\sqrt{4\pi}}\,.
$$
Let us fix $\varepsilon>0$. Taking advantage of  Trudinger-Moser
inequality and the fact that $\|f_\alpha\|_{L^2}\to 0$, we infer
\begin{eqnarray*}
\int\,\left({\rm e}^{(4\pi-\varepsilon)|f_\alpha(x)|^2}-1\right)\,dx&\leq&C_\varepsilon\,\|f_\alpha\|_{L^2}^2,\\
&\leq& \kappa,\quad\mbox{for}\quad \alpha\geq \alpha_\varepsilon\,.
\end{eqnarray*}
Hence, for any $\varepsilon>0$,
$$
\limsup_{\alpha\to\infty}\,\|f_\alpha\|_{\mathcal
L}\leq\frac{1}{\sqrt{4\pi-\varepsilon}},
$$
which ends the proof of the first assertion. To prove the second
one, let us write \beqn
\int\,\left({\rm e}^{\tiny{\frac{|f_\alpha(x)|^2}{\alpha^{1/2}}}}-1\right)\,dx&=&2\pi\int_0^{{\rm e}^{-\alpha}}\,\left({\rm e}^{\frac{\sqrt{\alpha}}{2\pi}}-1\right)r\,dr+2\pi\int_{{\rm e}^{-\alpha}}^1\left({\rm e}^{\frac{\log^2 r}{2\pi\alpha^{3/2}}}-1\right)r\,dr\\
&\leq&\pi\left({\rm e}^{\frac{\sqrt{\alpha}}{2\pi}}-1\right){\rm
e}^{-2\alpha}+2\pi\left(1-{\rm e}^{-\alpha}\right){\rm
e}^{\frac{\alpha^{1/2}}{2\pi}}\,. \eeqn This implies that, for
$\alpha$ small enough, $\|f_\alpha\|_{\cL}\leq \alpha^{1/4}$, which
leads to the result.
\end{proof}
The difference between the behavior of these families in Orlicz
space when $\alpha\to 0$ or $\alpha\to\infty$ comes from the fact
that the concentration effect is only displayed by this family when
$\alpha\to\infty$. Indeed, in the case where $\alpha\to\infty$ we
have the following result which does not occur when $\alpha\to 0$.

\begin{prop}
\label{concentration} ~$f_{\alpha}$ being the family of functions
defined above, we have
$$|\nabla f_{\alpha}|^2\to \delta(x=0)\quad\mbox{and}\quad {\rm e}^{4\pi|f_{\alpha}|^2}-1\to 2\pi\delta(x=0)\quad\mbox{as}\quad \alpha\to\infty\quad\mbox{in}\quad {\cD}'(\R^2)\,.$$
\end{prop}

\begin{proof}[Proof of Proposition \ref{concentration}]
 Straightforward  computations give for any  smooth compactly supported function $\varphi$
\beqn
\Int\,|\nabla f_{\alpha}(x)|^2\varphi(x)\,dx&=&\frac{1}{2\pi\alpha}\,\Int_{{\rm e}^{-\alpha}}^1\,\Int_0^{2\pi}\,\frac{\varphi(r\cos\theta,r\sin\theta)}{r}\,dr\,d\theta\\
&=&\varphi(0)+\frac{1}{2\pi\alpha}\,\Int_{{\rm
e}^{-\alpha}}^1\,\Int_0^{2\pi}\,\frac{\varphi(r\cos\theta,r\sin\theta)-\varphi(0)}{r}\,dr\,d\theta
\eeqn Since
$\Big|\frac{\varphi(r\cos\theta,r\sin\theta)-\varphi(0)}{r}\Big|\leq
\|\nabla\varphi\|_{L^\infty}$, we deduce that $\Int\,|\nabla
f_{\alpha}(x)|^2\varphi(x)\,dx\to \varphi(0)$ as $\alpha\to\infty$,
which ensures the result.  Similarly, we have

\beqn
\int\left({\rm e}^{4\pi|f_\alpha(x)|^2}-1\right)\varphi(x)dx&=&\int_0^{{\rm e}^{-\alpha}}\Int_0^{2\pi}\left({\rm e}^{2\alpha}-1\right)\varphi(r\cos\theta,r\sin\theta)rdrd\theta\\&+&\int_{{\rm e}^{-\alpha}}^1\int_0^{2\pi}\left({\rm e}^{\frac{2}{\alpha}\log^2 r}-1\right)\varphi(r\cos\theta,r\sin\theta)rdrd\theta\\
&=&\pi \varphi(0)\left(1-{\rm e}^{-2\alpha}\right)+2\pi\varphi(0)\int_{{\rm e}^{-\alpha}}^1\left({\rm e}^{\frac{2}{\alpha}\log^2 r}-1\right)r dr\\
&+&\int_0^{{\rm e}^{-\alpha}}\int_0^{2\pi}\left({\rm
e}^{2\alpha}-1\right)\left(\varphi(r\cos\theta,r\sin\theta)-\varphi(0)\right)r
dr d\theta\\&+&\int_{{\rm e}^{-\alpha}}^1\Int_0^{2\pi}\left({\rm
e}^{\frac{2}{\alpha}\log^2
r}-1\right)\left(\varphi(r\cos\theta,r\sin\theta)-\varphi(0)\right)r
dr d\theta. \eeqn We conclude by using the following lemma.
\end{proof}
\begin{lem}
\label{add1}
When $\alpha$ goes to infinity
\bq
\label{1}
I_\alpha:=\int_{{\rm e}^{-\alpha}}^1\;r\,{\rm e}^{\frac{2}{\alpha}\log^2 r}\,
dr\to 1
\eq
and
\bq
\label{13}
J_\alpha:=\int_{{\rm e}^{-\alpha}}^1\;r^2\,{\rm
e}^{\frac{2}{\alpha}\log^2 r}\, dr\,\to \frac{1}{3}.
\eq
\end{lem}
\begin{proof}[Proof of Lemma \ref{add1}]
The change of variable $y:=\sqrt{\frac{2}{\alpha}}\left(-\log r-\frac{\alpha}{2}\right)$ yields
$$
I_\alpha=2\sqrt{\frac{\alpha}{2}}\,{\rm e}^{-\frac{\alpha}{2}}\,\int_0^{\sqrt{\frac{\alpha}{2}}}\,{\rm e}^{y^2}\,dy\,.
$$
Taking advantage of the following obvious equivalence at infinity which can be derived by integration by parts
\bq
\label{equiva}
\int_0^{A}\,{\rm e}^{y^2}\,dy\sim \frac{{\rm e}^{A^2}}{2A},
\eq
we deduce \eqref{1}. \\ Similarly for the second term, the change of variable $y:=\sqrt{\frac{2}{\alpha}}\left(-\log r-\frac{3}{4}\alpha\right)$ implies that
\begin{eqnarray*}
J_\alpha&=&\sqrt{\frac{\alpha}{2}}\,{\rm e}^{-\frac{9}{8}\alpha}\,\int_{-\frac{3}{2}\sqrt{\frac{\alpha}{2}}}^{\frac{1}{2}\sqrt{\frac{\alpha}{2}}}\;{\rm e}^{y^2}\,dy\\&=&\sqrt{\frac{\alpha}{2}}\,{\rm e}^{-\frac{9}{8}\alpha}\,\int_0^{\frac{1}{2}\sqrt{\frac{\alpha}{2}}}\;{\rm e}^{y^2}\,dy+\sqrt{\frac{\alpha}{2}}\,{\rm e}^{-\frac{9}{8}\alpha}\,\int_0^{\frac{3}{2}\sqrt{\frac{\alpha}{2}}}\;{\rm e}^{y^2}\,dy.
\end{eqnarray*}
According to \eqref{equiva}, we get \eqref{13}.
\end{proof}

\begin{rem}
When $\alpha$ goes to zero, we get a spreading rather than a concentration.
Notice also  that for small  values of the function, our Orlicz space behaves like
$L^2$ (see Proposition \ref{Orl-L2}) and simple computations show that  $\|f_{\alpha}\|_{L^2(\R^2)}$ goes to zero when
$\alpha$ goes to zero.
\end{rem}
In fact, the conclusion of Proposition \ref{concentration} is
available for more general radial sequences. More precisely, we
have the following result due to P. -L.  Lions  (in a slightly different form):
\begin{prop}
\label{generalconcent} Let $(u_n)$ be a sequence in
$H^1_{rad}(\R^2)$ such that
$$
u_n\rightharpoonup 0\quad\mbox{in}\quad
H^1,\quad\liminf_{n\to\infty}\,\|u_n\|_{\cL}>0\quad\mbox{and}\quad\lim_{R\to\infty}\;
\limsup_{n\to\infty}\,\int_{|x|>R}\,|u_n(x)|^2\,dx=0\,.
$$
Then, there exists a constant $c>0$ such that
\bq
\label{concent}
|\nabla u_n(x)|^2\,dx\rightharpoondown\,\mu\geq c\,\delta(x=0)\quad (n\to\infty)
\eq
weakly in the sense of measures.
\end{prop}

\begin{rem}
The hypothesis of compactness at infinity
$$
\lim_{R\to\infty}\;
\limsup_{n\to\infty}\,\int_{|x|>R}\,|u_n(x)|^2\,dx=0
$$
is necessary to get  \eqref{concent}. For instance,
 $u_n(x)=\frac{1}{n}\,{\rm e}^{-|\frac{x}{n}|^2}$   satisfies $\|u_n\|_{L^2}=C>0$,
 $\|\nabla u_n\|_{L^2}\to 0$ and  $\liminf_{n\to\infty}\,\|u_n\|_{\cL}>0 $.
\end{rem}

\subsection{Fundamental Remark}
In order to describe  the lack of compactness of the Sobolev
embedding of~$H^1_{rad}(\R^2)$ in Orlicz space, we will make the
change of variable $s:=-\log r$, with $r=|x|$. We associate then to
any radial function $u$ on $\R^2$ a one space variable function $v$
defined by $v(s)=u({\rm e}^{-s}).$ It follows that \bq \label{prel1}
\|u\|_{L^2}^2=2\pi\int_{\R}\,|v(s)|^2{\rm e}^{-2s}\,ds, \eq

\bq \label{prel2} \|\nabla u\|_{L^2}^2=2\pi\int_{\R}\,|v'(s)|^2\,ds,
\quad \quad \mbox{and} \eq

\bq \label{prel3} \int_{\R^2}\,\left({\rm
e}^{|\frac{u(x)}{\lambda}|^2}-1\right)\,dx=2\pi\int_{\R}\,\left({\rm
e}^{|\frac{v(s)}{\lambda}|^2}-1\right){\rm e}^{-2s}\,ds. \eq

The starting point in our analysis is the following observation
related to the Lions's example
$$
\tilde{f_{\alpha}}(s):=f_{\alpha}({\rm
e}^{-s})=\sqrt{\frac{\alpha}{2\pi}}\,{\mathbf
L}\left(\frac{s}{\alpha}\right),
$$
where
\begin{eqnarray*}
{\mathbf L}(t)&=&\; \left\{
\begin{array}{cllll}0 \quad&\mbox{if}&\quad
t\leq 0,\\ t \quad
&\mbox{if}&\quad 0\leq t\leq 1 ,\\
1 \quad&\mbox{if}&t\geq 1.
\end{array}
\right.
\end{eqnarray*}
The sequence $\alpha\to\infty$ is called the scale and the function
${\mathbf L}$  the profile. In fact, the Lions's example generates
more elaborate situations which help us to understand the defect of
compactness of Sobolev embedding in Orlicz space. For example, it
can be seen that for the sequence $g_k:=f_k+f_{2k}$ we have
$g_k(s)=\sqrt{\frac{k}{2\pi}}\;{\psi}\left(\frac{s}{k}\right),$
where
\begin{eqnarray*}
\psi(t)&=&\; \left\{
\begin{array}{cllll}0 \quad&\mbox{if}&\quad
t\leq 0,\\ t+\frac{t}{\sqrt{2}} \quad
&\mbox{if}&\quad 0\leq t\leq 1 ,\\
1+\frac{t}{\sqrt{2}} \quad&\mbox{if}&1\leq t\leq 2,\\
1+\sqrt{2}\quad&\mbox{if}& t\geq 2.
\end{array}
\right.
\end{eqnarray*}
This is due to the fact that the scales $(k)_{k\in\N}$ and
$(2k)_{k\in\N}$ are not orthogonal (see Definition \ref{ortho}
below) and thus they give a unique profile. However, for the
sequence $h_k:=f_k+f_{k^2}$, the situation is completely different
and a decomposition under the form $$ h_k(x)\asymp
\sqrt{\frac{\alpha_k}{2\pi}}\;\psi\left(\frac{-\log
|x|}{\alpha_k}\right) $$ is not possible, where the symbol
$\asymp$ means that the difference is compact in the Orlicz space
${\cL}$. The reason behind is that the scales $(k)_{k\in\N}$ and
$(k^2)_{k\in\N}$ are orthogonal. \\ It is worth noticing that in
the above examples the support is a fixed compact, and thus at
first glance the construction cannot be adapted in the general
case. But as shown by the following example, no assumption on the
support  is needed to "display" lack of compactness in the Orlicz
space. Indeed, let $R_\alpha$ in $(0,\infty)$ such that
 \begin{equation}
 \label{Ralpha1}
 \frac{R_\alpha}{\sqrt{\alpha}}\to 0,\quad \alpha\to \infty,
 \end{equation}and
 \begin{equation}
 \label{Ralpha2}
 a:=\Liminf_{\alpha\to\infty}\,\Big(\frac{\log R_\alpha}{\alpha}\Big)>-\infty.
  \end{equation}
  We can take for instance $R_{\alpha}=\alpha^\theta$ with $\theta<1/2$ and then $a=0$, or $R_\alpha={\rm e}^{-\gamma \alpha}$ with $\gamma\geq 0$
  and then  $a=-\gamma$. Remark that Assumption \eqref{Ralpha1} implies that $a$ is always negative.
  Now, let us  define the sequence
  $g_\alpha(x):=f_\alpha(\frac{x}{R_\alpha})$. It is obvious that
  the family $g_\alpha$ is not uniformly supported in a fixed
  compact subset of $\R^2$, in the case when $R_{\alpha}=\alpha^\theta$ with $0 <
  \theta<1/2$. Now,
  arguing exactly as for Lions's example, we can easily show that
  $$
  \|g_\alpha\|_{L^2}\sim\frac{R_\alpha}{2\sqrt{\alpha}},\quad
   \|\nabla g_\alpha\|_{L^2}=1\quad\mbox{and}\quad \|g_\alpha\|_{\cL}^2\geq
   \frac{\alpha}{2\pi\log\left(1+\frac{\kappa}{\pi}\left(\frac{{\rm e}^\alpha}{R_\alpha}\right)^2\right)}.$$
 Hence, $g_\alpha\rightharpoonup 0$ in $H^1$ and
 $\Liminf_{\alpha\to\infty}\,\|g_\alpha\|_{{\cL}}>0$.\\

 Up to a subsequence extraction, straightforward computation yields  the strong convergence to zero in $H^1$ for the difference $f_\alpha-g_\alpha$, in the case
  when $a=0$, which implies that $g_\alpha(x)\asymp \sqrt{\frac{\alpha}{2\pi}}\,{\mathbf L}\Big(\frac{-\log|x|}{\alpha}\Big)$. However,
  in the case when $a<0$, the sequence $(g_\alpha)$ converges strongly to $f_\alpha(\frac{\cdot}{{\rm e}^{\alpha a}})$ in $H^1$ and then the profile is slightly different in the sense that $g_\alpha(x)\asymp \sqrt{\frac{\alpha}{2\pi}}\;{\mathbf L}_a\,\Big(\frac{-\log|x|}{\alpha}\Big)$ where ${\mathbf L}_a(s)={\mathbf L}(s+a)$.

To be more complete and in order to state our main result in a clear
way, let us  introduce some definitions as in \cite{Ge2} for
instance.
\begin{defi}
\label{ortho} A scale is a sequence $\underline{\alpha}:=(\alpha_n)$
of positive real numbers going to infinity. We shall say that two
scales ${\underline\alpha}$ and ${\underline\beta}$ are orthogonal (
in short $\underline{\alpha}\perp\underline{\beta}$) if
    $$
   \Big|\log\left({\beta_n}/{\alpha_n}\right)\Big|\to\infty.
    $$
\end{defi}
According to \eqref{prel1} and \eqref{prel2}, we introduce the
profiles as follows.
\begin{defi}
\label{Profi} The set of profiles is
$$
{\cP}:=\Big\{\;\psi\in L^2(\R,{\rm e}^{-2s}ds);\;\;\; \psi'\in
L^2(\R),\;\psi_{|]-\infty,0]}=0\,\Big\}.
$$
\end{defi}
\noindent Some remarks are in order:\\
{\bf a)} The limitation for scales tending to infinity  is justified by the behavior of $\|f_{\alpha}\|_{\mathcal L}$ stated in Proposition \ref{f_alpha}.\\
{\bf b)} The set ${\cP}$ is invariant under negative translations. More precisely, if $\psi\in{\cP}$ and $a\leq 0$ then $\psi_a(s):=\psi(s+a)$ belongs to ${\cP}$.\\
{\bf c)} It will be useful to observe that a profile (in the sense of Definition \ref{Profi}) is a continuous function since it belongs to $H^1_{loc}(\R)$.\\
{\bf d)} For a scale $\underline{\alpha}$ and a profile $\psi$,
define
$$g_{\underline{\alpha}, \psi}(x):=\sqrt{\frac{\alpha_n}{2\pi}}\;\psi\Big(\frac{-\log|x|}{\alpha_n}\Big)\,.$$
It is clear that, for any $\lambda>0$,
$$g_{\underline{\alpha}, \psi}=g_{\lambda\underline{\alpha}, \psi_{\lambda}},$$
where $ \psi_{\lambda}(t)=\frac{1}{\sqrt{\lambda}}\,\psi(\lambda
t)$.
\\
The next proposition illustrates the above definitions of scales and
profiles.
\begin{prop}
\label{single} Let $\psi\in{\cP}$ a profile, $(\alpha_n)$ any scale
and set
$$
g_n(x):=\sqrt{\frac{\alpha_n}{2\pi}}\;\psi\Big(\frac{-\log|x|}{\alpha_n}\Big).
$$
Then \bq \label{profile}
\frac{1}{\sqrt{4\pi}}\,\Sup_{s>0}\;\frac{|\psi(s)|}{\sqrt{s}}\leq
\Liminf_{n\to\infty}\,\|g_n\|_{{\cL}}\leq
\Limsup_{n\to\infty}\,\|g_n\|_{{\cL}}\leq
\frac{1}{\sqrt{4\pi}}\,\|\psi'\|_{L^2}. \eq
\end{prop}
\begin{proof}
A simple computation shows that
$$
\|g_n\|_{L^2}^2=\alpha_n^2\,\int_0^\infty\,|\psi(t)|^2\,{\rm
e}^{-2\alpha_n t}\,dt\quad\mbox{and}\quad \|\nabla
g_n\|_{L^2}=\|\psi'\|_{L^2(\R)}\,.
$$
In view of \eqref{2D-embed}, we have
$$
 \Limsup_{n\to\infty}\,\|g_n\|_{{\cL}}\leq\frac{1}{\sqrt{4\pi}}\,\|\psi'\|_{L^2}+\frac{1}{\sqrt{4\pi}}\,\limsup_{n\to\infty}\,\|g_n\|_{L^2}\,.
 $$
Therefore, to get the right hand side of \eqref{profile}, it
suffices to prove that the sequence $(g_n)$ converges strongly to
$0$ in $L^2$. To do so, let us first observe that
$$
\frac{\psi(t)}{\sqrt{t}}\to 0\quad\mbox{as}\quad t\to 0\,.
$$
Indeed
$$
|\psi(t)|=\Big|\int_0^t\,\psi'(\tau)\,d\tau\Big|\leq
\sqrt{t}\,\Big(\int_0^t\,|\psi'(\tau)|^2\,d\tau\Big)^{1/2},
$$
which ensures the result since $\psi'\in L^2(\R)$. Now,
$\varepsilon>0$ being fixed, we deduce the existence of $t_0>0$ such
that
$$
|\psi(t)|^2\leq \varepsilon\,t \quad\mbox{for}\quad 0<t<t_0\,.
$$
Therefore
$$
\alpha_n^2\,\int_0^{t_0}\,|\psi(t)|^2\,{\rm e}^{-2\alpha_n
t}\,dt\leq \varepsilon\,\int_0^\infty\,s{\rm
e}^{-2s}\,ds=\frac{\varepsilon}{4}\,.
$$
In other respects, by Lebesgue theorem
$$
\alpha_n^2\,\int_{t_0}^\infty\,|\psi(t)|^2\,{\rm e}^{-2\alpha_n
t}\,dt\to 0\quad\mbox{as}\quad n\to\infty\,.
$$
This leads to the strong convergence of $(g_n)$ to $0$ in $L^2$.

Let us now go to the proof of the left hand side inequality of
\eqref{profile}. Setting $L=\Liminf_{n\to\infty}\,\|g_n\|_{{\cL}}$,
we have  according to \eqref{norm1}  for fixed $\varepsilon>0$ and
$n$ large enough (up to a subsequence extraction)
$$
\int_{\R^2}\,\left({\rm
e}^{|\frac{g_n(x)}{L+\varepsilon}|^2}-1\right)\,dx\leq \kappa\,.
$$
A straightforward computation yields to
$$
\alpha_n\,\int_0^\infty\,{\rm e}^{2\alpha_n t
\Big(\frac{1}{4\pi(L+\varepsilon)^2}\left(\frac{\psi(t)}{\sqrt{t}}\right)^2-1\Big)}\,dt\leq
C,
$$
for some absolute constant $C$ and for $n$ large enough. Since $\psi(t)$ is continuous, we
deduce that necessarily, for any $t>0$,
$$
\frac{1}{\sqrt{4\pi}}\frac{|\psi(t)|}{\sqrt{t}}\,\leq
\,L+\varepsilon,
$$
which ensures the result.
\end{proof}
\subsection{Statement of the results}


Our first main goal is to establish that  the characterization  of
the lack of compactness of the embedding
$$
H^1_{rad}\hookrightarrow{\mathcal L},
$$
can be reduced to the Lion's example. More precisely, we shall prove
that the lack of compactness of this embedding can be described in
terms of an asymptotic decomposition  as follows:
\begin{thm}
\label{main} Let $(u_n)$ be a bounded sequence in
$H^1_{rad}(\R^2)$ such that \bq \label{main-assum1}
u_n\rightharpoonup 0, \eq \bq \label{main-assum2}
\limsup_{n\to\infty}\|u_n\|_{\mathcal L}=A_0>0, \quad \quad
\mbox{and} \eq \bq \label{main-assum3} \lim_{R\to\infty}\;
\limsup_{n\to\infty}\,\int_{|x|>R}\,|u_n|^2\,dx=0. \eq Then, there
exists a sequence $(\underline{\alpha}^{(j)})$ of pairwise
orthogonal scales and a sequence of profiles $(\psi^{(j)})$ in
${\cP}$ such that, up to a subsequence extraction, we have for all
$\ell\geq 1$, \bq \label{decomp}
u_n(x)=\Sum_{j=1}^{\ell}\,\sqrt{\frac{\alpha_n^{(j)}}{2\pi}}\;\psi^{(j)}\left(\frac{-\log|x|}{\alpha_n^{(j)}}\right)+{\rm
r}_n^{(\ell)}(x),\quad\limsup_{n\to\infty}\;\|{\rm
r}_n^{(\ell)}\|_{\mathcal
L}\stackrel{\ell\to\infty}\longrightarrow 0. \eq Moreover, we have
the following stability estimates \bq \label{ortogonal} \|\nabla
u_n\|_{L^2}^2=\Sum_{j=1}^{\ell}\,\|{\psi^{(j)}}'\|_{L^2}^2+\|\nabla
{\rm r}_n^{(\ell)}\|_{L^2}^2+\circ(1),\quad n\to\infty. \eq

\end{thm}
\begin{rems}\quad\\
{\bf a)} As in higher dimensions, the decomposition \eqref{decomp}
is not unique (see \cite{Ge2}).\\ {\bf b)} The assumption
\eqref{main-assum3} means that there is no lack of compactness at
infinity. It is in particularly satisfied when the
sequence~$(u_n)_{n\in \N}$  is supported in a fixed compact
of~$\R^2$ and also by $g_\alpha$. \\ {\bf c)} Also, this
assumption implies the condition $\psi_{|]-\infty,0]}=0$ included
in the definition of the set of profiles. Indeed, first let us
observe that under Condition \eqref{main-assum3}, necessarily each
element $ g_n^{(j)}(x):=
\sqrt{\frac{\alpha_n^{(j)}}{2\pi}}\;\psi^{(j)}\left(\frac{-\log|x|}{\alpha_n^{(j)}}\right)$
of Decomposition \eqref{decomp} does not display lack of
compactness at infinity.  The problem is then reduced to prove
that if a sequence $g_n=
\sqrt{\frac{\alpha_n}{2\pi}}\;\psi\left(\frac{-\log|x|}{\alpha_n}\right)$,
where $(\alpha_n)$ is any scale and $\psi\in L^2(\R,{\rm
e}^{-2s}ds)$ with $\psi'\in L^2(\R)$, satisfies Hypothesis
\eqref{main-assum3}  then consequently $\psi_{|]-\infty,0]}=0$.
Let us then consider a sequence $g_n$ satisfying the above
assumptions.
 This yields
$$
\lim_{R\to\infty}\;
\limsup_{n\to\infty}\,\Big(\alpha_n^2\,\int_{-\infty}^{-\frac{\log
R}{\alpha_n}}\,|\psi(t)|^2\,{\rm e}^{-2\alpha_n\,t}\,dt\Big)=0\,.
$$
Now, if $\psi(t_0)\neq 0$ for some $t_0<0$ then by continuity, we
get $|\psi(t)|\gtrsim 1$ for $t_0-\eta\leq t\leq t_0+\eta<0$.
Hence, for $n$ large enough,
$$
\alpha_n^2\,\int_{-\infty}^{-\frac{\log
R}{\alpha_n}}\,|\psi(t)|^2\,{\rm e}^{-2\alpha_n\,t}\,dt\gtrsim
\frac{\alpha_n}{2}\,\Big({\rm e}^{-2\alpha_n\,(t_0-\eta)}-{\rm
e}^{-2\alpha_n\,(t_0+\eta)}\Big),
$$
which leads easily to the desired result.\\
{\bf d)} Compared with the decomposition in \cite{Ge2}, it can be seen that there's no core in \eqref{decomp}. This is justified by the radial setting.\\
{\bf e)} The description of the lack of compactness of the embedding of $H^1(\R^2)$ into Orlicz space in the general frame is much harder than the radial setting. This will be dealt with in a forthcoming paper.\\
{\bf f)} Let us mention that M. Struwe in \cite{Struwe88} studied the loss of compactness for the functional
$$
E(u)=\frac{1}{|\Omega|}\,\Int_{\Omega}\;{\rm e}^{4\pi |u|^2}\,dx,
$$
where $\Omega$ is a bounded domain in $\R^2$.
\end{rems}

It should be emphasized that, contrary to the case of
Sobolev embedding in Lebesgue spaces, where the asymptotic
decomposition derived by P. G\'erard in \cite{Ge2} leads to $$
\|u_n\|^p_{L^p} \to \sum_{j\geq 1} \|\psi^{(j)}\|^p_{L^p},$$ Theorem
\ref{main} induces to
\begin{equation}
\label{OrliczMax} \|u_n\|_{\cL}\to \sup_{j\geq
1}\,\left(\lim_{n\to\infty}\,\|g_n^{(j)}\|_{\cL}\right),
\end{equation}
thanks to the following proposition.
\begin{prop}
\label{sumOrlicz} Let $(\underline{\alpha}^{(j)})_{1\leq
j\leq\ell}$ be a family of pairwise orthogonal scales and
$(\psi^{(j)})_{1\leq j\leq\ell}$ be a family of profiles, and set
$$
g_n(x)=\Sum_{j=1}^{\ell}\,\sqrt{\frac{\alpha_n^{(j)}}{2\pi}}\;\psi^{(j)}\left(\frac{-\log|x|}{\alpha_n^{(j)}}\right):=\Sum_{j=1}^{\ell}\,g_n^{(j)}(x)\;.
$$ Then
\begin{equation}
\label{OrliczMax1} \|g_n\|_{\cL}\to \sup_{1\leq
j\leq\ell}\,\left(\lim_{n\to\infty}\,\|g_n^{(j)}\|_{\cL}\right)\;.
\end{equation}
\end{prop}
A consequence of this proposition is that the first profile in the decomposition \eqref{decomp} can be chosen such that up to extraction
$$ \limsup_{n\to\infty}\|u_n\|_{\mathcal L}=A_0 = \lim_{n\to\infty}\left\|  \sqrt{\frac{\alpha_n^{(1)}}{2\pi}}\;\psi^{(1)}\left(\frac{-\log|x|}{\alpha_n^{(1)}}\right)  \right\|_{\mathcal L} \,.$$


\subsection{Structure of the paper} Our paper is organized as follows: we first describe in Section~\ref{secmain}
the algorithmic construction of the decomposition of a bounded
sequence~$(u_n)_{n\in \N}$  in $H^1_{rad}(\R^2)$, up to a
subsequence extraction, in terms of asymptotically orthogonal
profiles in the spirit of the
Lions'examples~$\sqrt{\frac{\alpha}{2\pi}}\;\psi(\frac{-\log|x|}{\alpha})$, and then prove Proposition~\ref{sumOrlicz}.
 Section~\ref{Wave}
is devoted to the study of nonlinear wave equations with
exponential growth, both in the sub-critical and critical cases.
The purpose is then to investigate the influence of the nonlinear
term on the main features of solutions of nonlinear wave equations
by comparing their evolution with the evolution of the solutions
of the Klein-Gordon equation. Finally, we deal in appendix with
several complements for the sake of completeness.

Finally, we mention that, $C$ will be used to denote a constant
which may vary from line to line. We also use $A\lesssim B$ to
denote an estimate of the form $A\leq C B$ for some absolute
constant $C$ and $A\approx B$ if $A\lesssim B$ and $B\lesssim A$.
For simplicity, we shall also still denote by $(u_n)$ any
subsequence of $(u_n)$.


\section{Extraction of scales and profiles}
\label{secmain}


This section is devoted to the proofs of Theorem \ref{main} and
Proposition \ref{sumOrlicz}. Our approach to extract scales and
profiles relies on a diagonal subsequence extraction and uses in a
crucial way the radial setting and particularly the fact that we
deal with bounded functions far away from the origin. The heart of
the matter is reduced to the proof of the following lemma.
\begin{lem}
\label{heart} Let $(u_n)$ be a sequence in $H^1_{rad}(\R^2)$
satisfying the assumptions of the Theorem \ref{main}. Then there
exists a scale $(\alpha_n)$ and a profile $\psi$ such that
\begin{equation}
\label{hearteq}
 \|\psi'\|_{L^2}\geq C\,A_0,
 \end{equation}
where $C$ is a universal constant.
 \end{lem}

Roughly speaking, the proof  is done in three steps. In the first
step, according to Lemma \ref{heart}, we extract  the first scale
and the first profile satisfying the condition \eqref{hearteq}. This
reduces the problem to the study of the remainder term. If the limit of its Orlicz norm is null we stop the process. If not, we prove
that this remainder term satisfies the same properties as the
sequence start which allow us to apply the lines of reasoning
 of the first step and extract a second  scale and a second
profile which verifies the above key property \eqref{hearteq}. By
contradiction arguments, we get the property of orthogonality
between the two first scales. Finally, we prove that this process converges.




\subsection{Extraction of the first scale and the first profile}

Let us consider a bounded sequence $(u_n)$ in $H^1_{rad}(\R^2)$
satisfying hypothesis
\eqref{main-assum1}-\eqref{main-assum2}-\eqref{main-assum3} and let
us set $v_n(s)=u_n({\rm e}^{-s})$. The following lemma summarize
some properties of the sequence $(u_n)$ that will be useful to
implement the proof strategy.
\begin{lem}
\label{basicproper} Under the above assumptions and up to a
subsequence extraction, the  sequence $(u_n)$  converges strongly to
$0$ in $L^2$ and we have, for any $M\in\R$, \bq \label{bound}
\|v_n\|_{L^\infty(]-\infty, M[)}\to 0,\qquad n\to \infty. \eq

\end{lem}
\begin{proof}
 Let us first observe that for any $R>0$, we have
$$ \|u_n\|_{L^2}=\|u_n\|_{L^2(|x|\leq R)}+\|u_n\|_{L^2(|x|> R)}\,.
$$ Now, by virtue of Rellich's theorem, the Sobolev space
$H^1(|x|\leq R)$ is compactly embedded in $L^2(|x|\leq  R)$.
Therefore, $$ \limsup_{n\to\infty}\,\|u_n\|_{L^2}\leq
\limsup_{n\to\infty}\|u_n\|_{L^2(|x|\geq R)}. $$ Taking advantage
of the compactness at infinity of the sequence, we deduce  the
strong convergence of the sequence $(u_n)$  to zero in $L^2$.\\ On
the other hand, Property \eqref{bound} derives immediately
from the boundedness of $(u_n)$ in $H^1$, the strong convergence
to zero of $(u_n)$ in $L^2$ and the following well known radial
estimate recalled in Lemma \ref{apa3} $$ |u(r)|\leq
\frac{C}{\sqrt{r}}\,\|u\|_{L^2}^{\frac{1}{2}}\|\nabla
u\|_{L^2}^{\frac{1}{2}}. $$
\end{proof}
The first step is devoted to the determination of the first scale
and the first profile.
\begin{prop}
\label{step1} For any $\de>0$, we have
\begin{equation}
\label{depart} \sup_{s\geq
0}\left(\Big|\frac{v_n(s)}{A_0-\de}\Big|^2-s\right)\to\infty,\quad
n\to\infty.
\end{equation}
\end{prop}
\begin{proof}
We proceed by contradiction. If not, there exists $\de>0$ such that,
up to a subsequence  extraction
\begin{equation}
\label{dominate} \sup_{s\geq 0,
n\in\N}\;\;\left(\Big|\frac{v_n(s)}{A_0-\de}\Big|^2-s\right)\leq
C<\infty.
\end{equation}

In one hand, thanks to \eqref{bound} and \eqref{dominate}, we get by
virtue of Lebesgue theorem
$$
\int_{|x|<1}\;\left({\rm
e}^{|\frac{u_n(x)}{A_0-\de}|^2}-1\right)\,dx=2\pi\,\int_{0}^\infty\;\left({\rm
e}^{|\frac{v_n(s)}{A_0-\de}|^2}-1\right)\,{\rm e}^{-2s}\,ds\to
0,\quad n\to\infty.
$$
On the other hand, using Lemma \ref{basicproper} and the simple fact
that for any positive $M$, there exists a finite constant $C_M$ such
that
$$
\sup_{|t|\leq M}\,\left(\frac{{\rm e}^{t^2}-1}{t^2}\right)<C_M,
$$
we deduce that
$$
\int_{|x|\geq 1}\;\left({\rm
e}^{|\frac{u_n(x)}{A_0-\de}|^2}-1\right)\,dx\leq C\|u_n\|_{L^2}^2\to
0\,.
$$
This leads finally to
$$
\int_{\R^2}\;\left({\rm
e}^{|\frac{u_n(x)}{A_0-\de}|^2}-1\right)\,dx\to 0,\quad n\to\infty.
$$
Hence
$$
\Limsup_{n\to\infty}\,\|u_n\|_{\mathcal L}\leq A_0-\de,
$$
which is in contradiction with Hypothesis \eqref{main-assum2}.
\end{proof}

\begin{cor}
\label{alphan1} Let us fix $\de= {A_0}/2$, then up to a subsequence
extraction, there exists a sequence $(\alpha_n^{(1)})$ in $\R_+$
tending to infinity such that
\begin{equation}
\label{est1}
4\,\Big|\frac{v_n(\alpha_n^{(1)})}{A_0}\Big|^2-\alpha_n^{(1)}\to\infty
\end{equation}
\end{cor}
\begin{proof}
Let us set
$$
W_n(s)=4\,\Big|\frac{v_n(s)}{A_0}\Big|^2-s,\quad\quad
a_n=\sup_{s}\;W_n(s).
$$
Then, there exists $\alpha_n^{(1)}>0$ such that
$$
W_n(\alpha_n^{(1)})\geq a_n-\frac{1}{n}.
$$
In other respects under \eqref{depart}, $a_n\to\infty$ and then
$W_n(\alpha_n^{(1)})\to\infty$. It remains to prove that
$\alpha_n^{(1)}\to\infty$. If not, up to a subsequence  extraction,
the sequence $(\alpha_n^{(1)})$ is bounded and so is
$(W_n(\alpha_n^{(1)}))$ by \eqref{bound}. This completes the proof.
\end{proof}
An immediate consequence of the previous corollary is the following
result
\begin{cor}
\label{2.4} Under the above hypothesis, we have
$$
\frac{ A_0}{2}\sqrt{\alpha_n^{(1)}}\leq| v_n(\alpha_n^{(1)})|\leq C
\sqrt{\alpha_n^{(1)}}+\circ(1),
$$
with  $ C= (\limsup_{n\to\infty}\,\|\nabla\,u_n\|_{L^2} )/
\sqrt{2\pi}$ and where, as in all that follows,  $\circ(1)$ denotes
a sequence which tends to $ 0 $ as $n$ goes to infinity.
\end{cor}
\begin{proof}
The left hand side inequality follows immediately from Corollary
\ref{alphan1}. In other respects, noticing that by virtue of
\eqref{bound}, the sequence $v_n(0)\to 0$, one can write for any
positive real $s$
$$
|v_n(s)|=\Big|v_n(0)+\Int_0^s\,v_n'(\tau)\,d\tau\Big|\leq
|v_n(0)|+s^{1/2}\|v_n'\|_{L^2}\leq  |v_n(0)|+s^{1/2}\,\frac{\|\na
u_n\|_{L^2}}{\sqrt{2\pi}},
$$
from which the right hand side of the desired inequality follows.
\end{proof}
Now we are able to extract the first profile. To do so, let us set
$$
\psi_n(y)=\sqrt{\frac{2\pi}{\alpha_n^{(1)}}}\;v_n(\alpha_n^{(1)} y).
$$
The following lemma summarize the principle properties of
$(\psi_n)$.
\begin{lem} Under notations of Corollary \ref{2.4}, there exists a
constant $C$ such that
$$
\frac{ A_0}{2}\sqrt{2\pi}\leq |\psi_n(1)|\leq  C +\circ(1).
$$
Moreover, there exists a profile $\psi^{(1)} \in{\cP}$  such that,
up to a subsequence extraction
$$
\psi_n'\rightharpoonup (\psi^{(1)})'\quad\mbox{in}\quad
L^2(\R)\quad\mbox{and}\quad \|(\psi^{(1)})'\|_{L^2}\geq
\frac{\sqrt{2\pi}}{2}\,A_0.
$$
\end{lem}
\begin{proof}
The first assertion is contained in Corollary \ref{2.4}. To prove
the second one, let us first remark that  since $\|\psi'_n\|_{L^2}=\|\nabla u_n\|_{L^2}$  then the sequence $(\psi_n')$ is bounded in
$L^2$. Thus, up to a
subsequence extraction, $(\psi_n')$ converges weakly in $L^2$ to
some function $g\in L^2$. In addition, $(\psi_n(0))$ converges in
$\R$ to $0$ and (still up to a subsequence  extraction)
$(\psi_n(1))$ converges in $\R$ to some constant $a$ satisfying
$|a|\geq \frac{ \sqrt{2\pi}}{2}\,A_0$. Let us then introduce the
function
$$
\psi^{(1)}(s):=\Int_0^s\,g(\tau)\,d\tau.
$$
Our task now is to show that $\psi^{(1)}$ belongs to the set
${\cP}$. Clearly $\psi^{(1)}\in {\mathcal C}(\R)$ and
$(\psi^{(1)})'=g\in L^2(\R)$. Moreover, since
$$
|\psi^{(1)}(s)|=\Big|\int_0^s\,g(\tau)\,d\tau\Big|\leq
s^{1/2}\,\|g\|_{L^2(\R)},
$$
we get $\psi^{(1)}\in L^2(\R^+, {\rm e}^{-2s}\,ds)$. It remains to
prove that $\psi^{(1)}(s)=0$ for all $s\leq 0$. Using the
boundedness of the sequence $(u_n)$ in $L^2(\R^2)$ and the fact that
$$
\|u_n\|_{L^2}^2=(\alpha_n^{(1)})^2\,\int_{\R}\, |\psi_n(s)|^2\,{\rm
e}^{-2\alpha_n^{(1)} s}\,ds,
$$
we deduce that
$$
\int_{-\infty}^0\,|\psi_n(s)|^2\,ds\leq
\frac{C}{(\alpha_n^{(1)})^2}.
$$
Hence, $(\psi_n)$ converges strongly to zero in $L^2(]-\infty, 0[)$,
and then almost everywhere (still up to a subsequence  extraction).
In other respects, since $(\psi'_n)$ converges weakly to $g$ in
$L^2(\R)$ and $\psi_n\in H^1_{loc}(\R)$, we infer that
$$
\psi_n(s)-\psi_n(0)=\int_0^s\,\psi'_n(\tau)\,d\tau\to
\int_0^s\,g(\tau)\,d\tau=\psi^{(1)}(s),
$$
from which it follows that
$$
\psi_n(s)\to \psi^{(1)}(s),\quad\mbox{for all}\quad s\in\R\,.
$$
As $\psi_n $ goes to zero for all $s \leq 0$, we deduce that $\psi^{(1)}(s)=0$ for all $s\leq 0$. Finally, we
have proved that $\psi^{(1)}\in {\cP}$ and $|\psi^{(1)}(1)|=|a|\geq
\frac{\sqrt{2\pi}}{2}\,A_0$. The fact that
$$
| \psi^{(1)}(1)|=\Big|\Int_0^1\,(\psi^{(1)})'(\tau)\,d\tau\Big|\leq
\|(\psi^{(1)})'\|_{L^2},
$$
yields $\|(\psi^{(1)})'\|_{L^2}\geq \frac{\sqrt{2\pi}}{2}\,A_0$.
\end{proof}
Set
 \begin{equation}
 \label{reste1}
{\rm
r}^{(1)}_n(x)=\sqrt{\frac{\alpha_n^{(1)}}{2\pi}}\;\left(\psi_n\left(\frac{-\log|x|}{\alpha_n^{(1)}}\right)-\psi^{(1)}\left(\frac{-\log|x|}{\alpha_n^{(1)}}\right)\right).
\end{equation}
It can be easily seen that
$$
\|\nabla\,{\rm
r}^{(1)}_n\|_{L^2(\R^2)}^2=\|\psi'_n-(\psi^{(1)})'\|_{L^2(\R)}^2\,.
$$
Taking advantage of  the fact that $(\psi'_n)$ converges weakly in
$L^2(\R)$ to $(\psi^{(1)})'$, we get the following result
\begin{lem}
\label{Concl1} Let $(u_n)$ be a sequence in $H^1_{rad}(\R^2)$
satisfying the assumptions of the Theorem \ref{main}. Then there
exists a scale $(\alpha_n^{(1)})$ and a profile $\psi^{(1)}$ such
that
$$
 \|(\psi^{(1)})'\|_{L^2}\geq \frac{\sqrt{2\pi}}{2}\,A_0,
 $$
and
 \begin{equation}
 \label{norme1}
 \limsup_{n\to\infty}\,\|\nabla\,{\rm r}^{(1)}_n\|_{L^2}^2\leq \limsup_{n\to\infty}\,\|\nabla\,u_n\|_{L^2}^2- \|(\psi^{(1)})'\|_{L^2}^2\,.
 \end{equation}
 where ${\rm
r}^{(1)}_n$ is given by \eqref{reste1}.
 \end{lem}

\subsection{Conclusion}
Our concern now is to iterate the previous process and to prove that
the algorithmic construction converges. Observing that, for $R\geq
1$, and thanks to the fact that $\psi^{(1)}_{|]-\infty,0]}=0$,
\begin{eqnarray*}
\|{\rm r}_n^{(1)}\|_{L^2(|x|\geq R)}^2&=&(\alpha_n^{(1)})^2\,\int_{-\infty}^{-\frac{\log R}{\alpha_n^{(1)}}}\,|\psi_n(t)-\psi^{(1)}(t)|^2\,{\rm e}^{-2\alpha_n^{(1)} t}\,dt,\\
&=&(\alpha_n^{(1)})^2\,\int_{-\infty}^{-\frac{\log R}{\alpha_n^{(1)}}}\,|\psi_n(t)|^2\,{\rm e}^{-2\alpha_n^{(1)} t}\,dt,\\
&=&\|u_n\|_{L^2(|x|\geq R)}^2,
\end{eqnarray*}
we deduce that $({\rm r}_n^{(1)})$ satisfies the hypothesis of
compactness at infinity \eqref{main-assum3}. This leads, according
to \eqref{norme1}, that $({\rm r}_n^{(1)})$ is bounded in
$H^1_{rad}$ and satisfies \eqref{main-assum1}.

 Let us define $A_1=\limsup_{n\to\infty}\,\|{\rm r}_n^{(1)}\|_{\cL}$. If $A_1=0$, we stop the process. If not, we apply the above argument
 to $ {\rm r}_n^{(1)}$  and then there exists a scale $(\alpha_n^{(2)})$ satisfying the statement of Corollary \ref{alphan1}
 with $A_1$ instead of $A_0$. In particular, there exists a constant
  $C$ such that
 \begin{equation}
 \label{tilde}
 \frac{ A_1}{2}\sqrt{\alpha_n^{(2)}}\leq |\tilde{\rm r}_n^{(1)}(\alpha_n^{(2)})|\leq
C \sqrt{\alpha_n^{(2)}}+\circ(1),
\end{equation}
 where $\tilde{\rm r}_n^{(1)}(s)={\rm r}_n^{(1)}({\rm e}^{-s})$.
Moreover, we claim that $\alpha_n^{(2)}\perp\alpha_n^{(1)}$, or
equivalently that
$\log\left|\frac{\alpha_n^{(2)}}{\alpha_n^{(1)}}\right|\to\infty$.
Otherwise, there exists a constant $C$ such that
$$
\frac{1}{C}\leq\left|\frac{\alpha_n^{(2)}}{\alpha_n^{(1)}}\right|\leq
C.
$$
Now, according to \eqref{reste1}, we have
$$
\tilde{\rm r}_n^{(1)}(\alpha_n^{(2)})=
\sqrt{\frac{\alpha_n^{(1)}}{2\pi}}\;\left(\psi_n\left(\frac{\alpha_n^{(2)}}{\alpha_n^{(1)}}\right)-\psi^{(1)}\left(\frac{\alpha_n^{(2)}}{\alpha_n^{(1)}}\right)\right).
$$
This yields a contradiction in view of \eqref{tilde} and the
following convergence result (up to a subsequence extraction)
$$
\psi_n\left(\frac{\alpha_n^{(2)}}{\alpha_n^{(1)}}\right)-\psi^{(1)}\left(\frac{\alpha_n^{(2)}}{\alpha_n^{(1)}}\right)\to
0.
$$
Moreover, there exists a profile $\psi^{(2)}$ in ${\cP}$ such that
$$
{\rm
r}_n^{(1)}(x)=\sqrt{\frac{\alpha_n^{(2)}}{2\pi}}\;\psi^{(2)}\left(\frac{-\log|x|}{\alpha_n^{(2)}}\right)+{\rm
r}_n^{(2)}(x),
$$
with  $\|(\psi^{(2)})'\|_{L^2}\geq \frac{\sqrt{2\pi}}{2}\,A_1$ and
$$
\limsup_{n\to\infty}\,\|\nabla\,{\rm r}^{(2)}_n\|_{L^2}^2\leq
\limsup_{n\to\infty}\,\|\nabla\,{\rm r}_n^{(1)}\|_{L^2}^2-
\|(\psi^{(2)})'\|_{L^2}^2\,.
$$
This leads to the following crucial estimate
$$
\limsup_{n\to\infty}\,\|{\rm r}^{(2)}_n\|_{H^1}^2\leq
C-\frac{\sqrt{2\pi}}{2}\,A_0^2-\frac{\sqrt{2\pi}}{2}\,A_1^2\,,
$$
with $ C= \limsup_{n\to\infty}\,\|\nabla\,u_n\|_{L^2}^2$.

At iteration $\ell$, we get
$$
u_n(x)=\Sum_{j=1}^{\ell}\,\sqrt{\frac{\alpha_n^{(j)}}{2\pi}}\;\psi^{(j)}\left(\frac{-\log|x|}{\alpha_n^{(j)}}\right)+{\rm
r}_n^{(\ell)}(x),
$$
with
$$
\limsup_{n\to\infty}\,\|{\rm r}^{(\ell)}_n\|_{H^1}^2\lesssim
1-A_0^2-A_1^2-\cdots -A_{\ell-1}^2\,.
$$
Therefore $A_\ell\to 0$ as $\ell\to\infty$ and the proof of the decomposition \eqref{decomp} is achieved.

The stability estimate \eqref{ortogonal} is a simple
consequence of the fact that for every $f, g\in L^2(\R)$ and every
orthogonal scales $\underline{\alpha}$ and $\underline{\beta}$, it
holds
$$
\int_{\R}\,\frac{1}{\sqrt{\alpha_n\beta_n}}\,f\left(\frac{s}{\alpha_n}\right)\,g\left(\frac{s}{\beta_n}\right)\,ds\to
0\quad\mbox{as}\quad n\to\infty\,.
$$
This ends the proof of the theorem.\\

Let us now go to the proof of Proposition \ref{sumOrlicz}.

\begin{proof}[Proof of Proposition \ref{sumOrlicz}]
We restrict ourselves to the example
$h_\alpha:=a\,f_\alpha+b\,f_{\alpha^2}$ where $a, b$ are two real
numbers. The general case is similar except for more technical
complications. Set $M:=\sup(|a|,|b|)$. We want to show that
$$
\|h_\alpha\|_{\cL}\to \frac{M}{\sqrt{4\pi}}\quad\mbox{as}\quad
\alpha\to\infty\,.
$$
We start by proving that
\bq
\label{Inf}
\liminf_{\alpha\to\infty}\,\|h_\alpha\|_{\cL}\geq\frac{M}{\sqrt{4\pi}}.
\eq

Let $\lambda>0$ such that
$$
\int_{\R^2}\,\left({\rm
e}^{\frac{|h_\alpha(x)|^2}{\lambda^2}}-1\right)\,dx\,\leq\,\kappa.
$$
This implies

\bq
\label{inf1}
\int_0^{{\rm e}^{-\alpha^2}}\; \left({\rm
e}^{\frac{|h_\alpha(r)|^2}{\lambda^2}}-1\right)\,r\,dr\,\leq\,\frac{\kappa}{2\pi},
\eq
and
\bq
\label{inf2}
\int_{{\rm e}^{-\alpha^2}}^{{\rm e}^{-\alpha}}\; \left({\rm
e}^{\frac{|h_\alpha(r)|^2}{\lambda^2}}-1\right)\,r\,dr\,\leq\,\frac{\kappa}{2\pi}.
\eq
Since
\begin{eqnarray*}
h_\alpha(r)&=&\; \left\{
\begin{array}{cllll}a\sqrt{\frac{\alpha}{2\pi}}+b\frac{\alpha}{\sqrt{2\pi}} \quad&\mbox{if}&\quad
r\leq {\rm e}^{-\alpha^2},\\\\ a\sqrt{\frac{\alpha}{2\pi}}-\frac{b}{\alpha\sqrt{2\pi}}\,\log r \quad
&\mbox{if}&\quad {\rm e}^{-\alpha^2}\leq r\leq {\rm e}^{-\alpha},
\end{array}
\right.
\end{eqnarray*}
we get from \eqref{inf1} and \eqref{inf2}
$$
\lambda^2\geq
\frac{\alpha\left(a+b\sqrt{\alpha}\right)^2}{2\pi\log\left(1+C{\rm
e}^{2\alpha^2}\right)}=\frac{b^2}{4\pi}+\circ(1),
$$
and
$$
\lambda^2\geq \frac{a^2\alpha+2 a
b\sqrt{\alpha}+b^2}{2\pi\log\left(1+C{\rm e}^{2\alpha}\right)}=\frac{a^2}{4\pi}+\circ(1)\,.
$$
This leads to \eqref{Inf} as desired.\\
In the general case, we have to replace  \eqref{inf1} and \eqref{inf2} by $\ell$ estimates
of that type. Indeed, assuming that $ \frac{\alpha_n^{(j)}}{ \alpha_n^{(j+1)} }  \to 0 $ when
n goes to infinity for $j = 1, 2, ..., l-1$, we replace \eqref{inf1} and \eqref{inf2}  by the fact that
\bq
\label{inf3}
\int_{{\rm e}^{- \alpha_n^{(j+1)}  }}^{{\rm e}^{- \alpha_n^{(j)}  }}  \; \left({\rm
e}^{\frac{|h_\alpha(r)|^2}{\lambda^2}}-1\right)\,r\,dr\,\leq\,\frac{\kappa}{2\pi},
\quad \quad j = 1, ..., l-1   \eq
and
\bq  \label{inf4}
\int_0^{{\rm e}^{- \alpha_n^{(l)}  }}  \; \left({\rm
e}^{\frac{|h_\alpha(r)|^2}{\lambda^2}}-1\right)\,r\,dr\,\leq\,\frac{\kappa}{2\pi}.
  \eq

Our concern now is to prove the second (and more difficult) part, that is
\bq
\label{Sup}
\limsup_{\alpha\to\infty}\,\|h_\alpha\|_{\cL}\leq\frac{M}{\sqrt{4\pi}}\,.
\eq
To do so, it is sufficient to show that for any $\eta>0$ small enough and $\alpha$ large enough
\bq
\label{sup1}
\int_{\R^2}\,\left({\rm
e}^{\frac{4\pi-\eta}{M^2}|h_\alpha(x)|^2}-1\right)\,dx\,\leq\,\kappa.
\eq
Actually, we will prove that the left hand side of  \eqref{sup1} goes to zero when
$\alpha$ goes to infinity.
We shall make use of the following lemma.
\begin{lem}
\label{p,q}
Let $p, q$ be two real numbers such that $0<p, q<2$. Set
\bq
\label{Ialpha}
{\mathbf I}_{\alpha}={\rm e}^{p\alpha}\;\int_{{\rm e}^{-\alpha^2}}^{{\rm e}^{-\alpha}}\;{\rm e}^{q\frac{\log^2 r}{\alpha^2}}\;r\,dr\,.
\eq
Then $I_\alpha\to 0$ as $\alpha\to\infty$.
\end{lem}
\begin{proof}[Proof of Lemma \ref{p,q}]
The change of variable $y=\frac{\sqrt{q}}{\alpha}\left(-\log r-\frac{\alpha^2}{q}\right)$ yields
$$
{\mathbf I}_{\alpha}=\alpha\,{\rm e}^{p\alpha}\frac{{\rm
e}^{-\frac{\alpha^2}{q}}}{\sqrt{q}}\,\int_{\sqrt{q}-\frac{\alpha}{\sqrt{q}}}^{\alpha\frac{q-1}{\sqrt{q}}}\;{\rm
e}^{y^2}\;dy\,.
$$
Since $\int_0^A\,{\rm e}^{y^2}\;dy\leq \frac{{\rm
e}^{A^2}}{A}$ for every nonnegative real $A$, we get (for $q>1$ for example)
$$
{\mathbf I}_{\alpha}\lesssim {\rm e}^{p\alpha}\,{\rm e}^{(q-2)\alpha^2}+{\rm e}^{(p-2)\alpha},
$$
and the conclusion follows. We argue similarly if $q\leq 1$.
\end{proof}
We return now to the proof of \eqref{Sup}. To this end, write
\begin{eqnarray}
\frac{(4\pi-\eta)}{M^2}\,|h_\alpha|^2&=&\frac{4\pi-\eta}{M^2}\,a^2\, f_\alpha^2+\frac{4\pi-\eta}{M^2}\,b^2\, f_{\alpha^2}^2+2\frac{4\pi-\eta}{M^2}\,ab\, f_\alpha\,f_{\alpha^2}\\
&:=&A_\alpha+B_\alpha+C_\alpha. \label{ABC}
\end{eqnarray}
The simple observation
\begin{eqnarray*}
 {\rm e}^{x+y+z}-1&=& \left({\rm
e}^{x}-1\right)\left({\rm e}^{y}-1\right)\left({\rm
e}^{z}-1\right)+\left({\rm e}^{x}-1\right)\left({\rm
e}^{y}-1\right)\\ \nonumber&+&\left({\rm e}^{x}-1\right)\left({\rm
e}^{z}-1\right)+\left({\rm e}^{y}-1\right)\left({\rm
e}^{z}-1\right)\\ \nonumber&+&\left({\rm e}^{x}-1\right)+\left({\rm
e}^{y}-1\right)+\left({\rm e}^{z}-1\right),
\end{eqnarray*}
 yields
\begin{eqnarray}
\label{Obs}
\nonumber
\int_{\R^2}\,\left({\rm
e}^{\frac{4\pi-\eta}{M^2}|h_\alpha(x)|^2}-1\right)\,dx\,&=&\,\int\,\left({\rm
e}^{A_\alpha}-1\right)\left({\rm e}^{B_\alpha}-1\right)\left({\rm
e}^{C_\alpha}-1\right)+\int\,\left({\rm
e}^{A_\alpha}-1\right)\left({\rm
e}^{B_\alpha}-1\right)\\&+&\int\,\left({\rm
e}^{A_\alpha}-1\right)\left({\rm
e}^{C_\alpha}-1\right)+\int\,\left({\rm
e}^{B_\alpha}-1\right)\left({\rm
e}^{C_\alpha}-1\right)\\ \nonumber&+&\int\,\left({\rm
e}^{A_\alpha}-1\right)+\int\,\left({\rm
e}^{B_\alpha}-1\right)+\int\,\left({\rm e}^{C_\alpha}-1\right)\,.
\end{eqnarray}

 By Trudinger-Moser estimate \eqref{Mos1}, we have for $\varepsilon\geq 0$ small enough,
 \bq
 \label{cv1}
 \Big\|{\rm e}^{A_\alpha}-1\Big\|_{L^{1+\varepsilon}}+\Big\|{\rm e}^{B_\alpha}-1\Big\|_{L^{1+\varepsilon}}\to 0\quad\mbox{as}\quad \alpha\to\infty.
 \eq

 To check that the last term in \eqref{Obs} tends to zero,  we use Lebesgue Theorem in the region ${\rm e}^{-\alpha}\leq r\leq 1$. Observe that one can replace $C_\alpha$ with $\gamma\;C_\alpha$ for
any $\gamma>0$.\\

The two terms containing both $A_\alpha$ and $C_\alpha$ or $B_\alpha$ and $C_\alpha$ can be handled in a similar way. Indeed,  by H\"older inequality and \eqref{cv1}, we infer (for $\varepsilon>0$  small enough)
\begin{eqnarray*}
\int\,\left({\rm e}^{A_\alpha}-1\right)\left({\rm
e}^{C_\alpha}-1\right)+\int\,\left({\rm e}^{B_\alpha}-1\right)\left({\rm
e}^{C_\alpha}-1\right)&\leq& \Big\|{\rm
e}^{A_\alpha}-1\Big\|_{L^{1+\varepsilon}}\;\Big\|{\rm
e}^{C_\alpha}-1\Big\|_{L^{1+\frac{1}{\varepsilon}}}\\&+&\Big\|{\rm
e}^{B_\alpha}-1\Big\|_{L^{1+\varepsilon}}\;\Big\|{\rm
e}^{C_\alpha}-1\Big\|_{L^{1+\frac{1}{\varepsilon}}}\to 0\;.
\end{eqnarray*}

Now, we claim that
\bq \label{diff} \int\,\left({\rm e}^{A_\alpha}-1\right)\left({\rm
e}^{B_\alpha}-1\right)\to 0,\;\;\alpha\to\infty\,. \eq

The main difficulty in the proof of \eqref{diff} comes from the term
$$
\int_{{\rm e}^{-\alpha^2}}^{{\rm e}^{-\alpha}}\,\left({\rm
e}^{A_\alpha(r)}-1\right)\left({\rm
e}^{B_\alpha(r)}-1\right)\,r\,dr\lesssim\;{\mathbf I}_{\alpha}:=
{\rm e}^{\frac{4\pi-\eta}{M^2}\,a^2\frac{\alpha}{2\pi}}\;\int_{{\rm
e}^{-\alpha^2}}^{{\rm e}^{-\alpha}}\,{\rm
e}^{\frac{4\pi-\eta}{M^2}\,b^2\,\frac{\log^2r}{2\pi\alpha^2}}\,r\,dr.
$$

Setting $p:=\frac{4\pi-\eta}{2\pi}\,\frac{a^2}{M^2}$ and $q:=\frac{4\pi-\eta}{2\pi}\,\frac{b^2}{M^2}$, we conclude thanks to Lemma \ref{p,q} since $0<p,q<2$. It easy to see that \eqref{diff} still holds if $A_\alpha$ and $B_\alpha$ are replaced by $(1+\varepsilon)A_\alpha$ and $(1+\varepsilon)B_\alpha$ respectively, where $\varepsilon\geq 0$ is small.

 Finally, for the first term in \eqref{Obs}, we use H\"older inequality and \eqref{diff}.\\Consequently, we obtain
$$
\limsup_{\alpha\to\infty}\,\|h_\alpha\|_{\cL}\leq\frac{M}{\sqrt{4\pi}}\,.
$$
 In the general case we replace \eqref{ABC}, by $\ell + \frac{\ell(\ell-1)}{2}$ terms and the rest of the proof is very similar. This completes the proof of Proposition \ref{sumOrlicz}.
\end{proof}



\section{Qualitative study of nonlinear wave equation}
\label{Wave}


This section is devoted to the qualitative study of the solutions of
the two-dimensional nonlinear Klein-Gordon  equation
\begin{equation}
\label{NLKG} \square u +u+f(u)=0,\quad
 u:\R_t\times\R_x^2\to\R,
 \end{equation}
where $$ f(u)=u\;\left({\rm e}^{4\pi u^2}-1\right). $$
Exponential type nonlinearities have been considered in several physical models
(see e.g.~\cite{LLT} on a model of self-trapped beams in plasma). For decreasing exponential nonlinearities, T. Cazenave in \cite{Caz} proved global well-posedness together with  scattering in the case of NLS.

It is known
(see \cite{NO-Wave, NO-DCDS}) that the Cauchy problem associated
to Equation (\ref{NLKG}) with Cauchy data small enough in
$H^1\times L^2$ is globally well-posed. Moreover, subcritical,
critical and supercritical regimes in the energy space are
identified (see \cite{2Dglobal}). Global well-posedness is
established in both subcritical and critical regimes while
well-posedness fails to hold in the supercritical one (we refer to
 \cite{2Dglobal} for more details). Very recently, M. Struwe \cite{Struwe09} has constructed global smooth solutions for the 2D energy critical wave equation with radially symmetric data. Although the techniques are different, this result might be seen as an analogue of Tao's result \cite{tao} for the 3D energy supercritical wave equation. Let us emphasize that  the solutions
 of the two-dimensional nonlinear Klein-Gordon  equation
 formally satisfy the conservation of energy
\begin{eqnarray}
\label{energy} E(u,t)&=&\|\partial_t u(t)\|_{L^2}^2+\|\nabla
u(t)\|_{L^2}^2+ \frac{1}{4\pi}\; \|{\rm e}^{4\pi u(t)^2}-1\|_{L^1}\\
\nonumber &=&E(u,0):=E_0.
\end{eqnarray}
The notion of criticality here depends on the size  of the initial
energy{\footnote{This is in contrast with higher dimensions where the
criticality depends on the nonlinearity.}} $E_0$ with respect to
$1$. This relies on the so-called Trudinger-Moser type inequalities
stated in Proposition \ref{mostrud} (see \cite{DlogSob} and
references therein for more details). Let us now precise the notions
of these regimes:
\begin{defi}
\label{d1} The Cauchy problem associated to Equation (\ref{NLKG})
 with initial data $(u_0,u_1)\in H^1(\R^2)\times L^2(\R^2)$ is
said to be {\it subcritical} if
$$
E_0<1.
$$
It is said {\it critical} if $E_0=1$ and {\it supercritical} if
$E_0>1$.
\end{defi}
It is then natural to investigate the feature of solutions of the
two-dimensional nonlinear Klein-Gordon  equation taking into
account  the different regimes, as in earlier works of P. G\'erard
\cite{Ge1} and H. Bahouri-P. G\'erard \cite{BG}. The approach that
we adopt here is the one introduced by P. G\'erard in \cite{Ge1}
which consists to compare the evolution of oscillations and
concentration effects displayed by sequences of solutions of the
nonlinear Klein-Gordon  equation (\ref{NLKG}) and solutions of the
linear Klein-Gordon equation. \beq \label{LKG} \square v+v=0 \eeq
More precisely, let $(\varphi_n, \psi_n)$ be a sequence of data in
$H^1\times L^2$ supported in some fixed ball and satisfying
\begin{equation}
 \label{weak-conv}
 \varphi_n\rightharpoonup 0\quad\mbox{in}\; H^1,\quad\psi_n\rightharpoonup 0
 \quad\mbox{in}\; L^2,\end{equation}
such that
\begin{equation}
 \label{subcrit}
 E^n\leq 1,\quad n\in \N
\end{equation}
where $E^n$ stands for the energy of $(\varphi_n, \psi_n)$ given
by $$ E^n=\|\psi_n\|_{L^2}^2+\|\nabla\varphi_n\|_{L^2}^2+
\frac{1}{4\pi}\; \|{\rm e}^{4\pi \varphi_n^2}-1\|_{L^1} $$ and let
us  consider $(u_n)$ and $(v_n)$ the sequences of finite energy
solutions of (\ref{NLKG}) and (\ref{LKG}) such that $$ (u_n,
\partial_t u_n)(0)=(v_n, \partial_t v_n)(0)=(\varphi _n,\psi_n).
$$ Arguing as in \cite{Ge1}, we introduce the following definition
\begin{defi}
Let $T$ be a positive time. We shall say that the sequence $(u_n)$
is linearizable on $[0,T]$, if $$
\Sup_{t\in[0,T]}E_c(u_n-v_n,t)\longrightarrow 0\quad\mbox{as}\quad
n\rightarrow\infty $$ where $E_c(w,t)$ denotes the kinetic energy
defined by:
$$E_c(w,t)=\Int_{{\mathbb{R}}^2}\left[|\partial_t
w|^2+|\nabla_x w|^2+|w|^2\right](t,x)\;dx.
$$
\end{defi}
The following results illustrate the critical feature of the
condition $E_0=1$.
\begin{thm}
\label{Subcrit} Under the above notations, let us assume that
$\Limsup_{n\rightarrow\infty}\;E^n<1$. Then, for every positive time
$T$, the sequence $(u_n)$ is linearizable on $[0,T]$.
\end{thm}
\begin{rem}
Let us recall that in the case of dimension $d \geq 3$, the same
kind of result holds. More precisely,  P. G\'erard proved in
\cite{Ge1}  that in the subcritical case, the nonlinearity does
not induce any effect on the behavior of the solutions.
\end{rem}
In the critical case i.e. $\Limsup_{n\rightarrow\infty}\;E^n=1$, it
turns out that a nonlinear effect can be produced and we have the
following result:
\begin{thm}
\label{Crit} Assume that $\Limsup_{n\rightarrow\infty}\;E^n=1$ and
let $T>0$. Then the sequence  $(u_n)$ is linearizable on $[0,T]$
provided that the sequence $(v_n)$ satisfies \beq
\label{crit-cond}\limsup_{n\to\infty}\;\|v_n\|_{L^\infty([0,T]; {\cL})}<\frac{1}{\sqrt{4\pi}}\,.
\eeq
\end{thm}
\begin{rem}
In Theorem \ref{Crit}, we give a sufficient condition on the
sequence $(v_n)$ with ensures the linearizability of the sequence
$(u_n)$. Similarly to higher dimensions, this condition concerns
the solutions of linear Klein-Gordon  equation. However, unlike to
higher dimensions, we are not able to prove the converse, that is
if the sequence $(u_n)$ is linearizable on $[0,T]$ then
$\limsup_{n\to\infty}\;\|v_n\|_{L^\infty([0,T];{\cL})}<\frac{1}{\sqrt{4\pi}}$.
The main difficulty in our approach is that we do not know whether $$
 \limsup_{n\to\infty}\;\|v_n\|_{L^\infty([0,T];{\cL})}=\frac{1}{\sqrt{4\pi}}\quad\mbox{and}\quad \|f(v_n)\|_{L^1([0,T]; L^2(\R^2))}\to 0\,.
 $$
\end{rem}
Before going into the proofs of Theorems \ref{Subcrit} and
\ref{Crit}, let us recall some well known and useful tools. The main
basic tool that we shall deal with is Strichartz estimate.


\subsection{Technical tools}
\subsubsection{Strichartz estimate}
Let us first begin by introducing the definition of  admissible
pairs.
\begin{defi}
\label{admiss} Let $\rho\in\R$. We say that
$(q,r)\in[4,\infty]\times[2,\infty]$ is a $\rho$-admissible pair
if \bq \label{Adm} \frac{1}{q}+\frac{2}{r}=\rho\,. \eq When
$\rho=1$, we shall say admissible instead of $1$-admissible.
\end{defi}
For example, $(4,\infty)$ is a $1/4$-admissible pair, and for
every $0<\varepsilon\leq 1/3$, the couple $(1+1/\varepsilon,
2(1+\varepsilon))$ is an admissible pair. The following Strichartz
inequalities that can be for instance found in \cite{NO-KG} will
be of constant use in what follows.
\begin{prop}[{\bf Strichartz estimate}]
\label{stri} Let $\rho\in\R$, $(q,r)$ a $\rho$-admissible pair and
$T>0$. Then
\begin{equation}
\label{e6} \|v\|_{L^q([0,T]; {\mathrm
B}^{\rho}_{r,2}(\R^2))}\lesssim\left[\|\partial_tv(0,\cdot)\|_{L^2(\R^2)}+\|v(0,\cdot)\|_{H^1(\R^2)}+\|\square
v+v\|_{L^1([0,T];L^2(\R^2))}\right],
\end{equation}
where ${\mathrm B}^{\rho}_{r,2}(\R^2)$ stands for the usual
inhomegenous Besov space (see for example \cite{chemin} or
\cite{RS} for a detailed exposition  on Besov spaces).
\end{prop}
Now, for any time slab $I\subset\R$, we shall denote
$$
\|v\|_{\mbox{\tiny ST}(I)}:=\sup_{(q,r)\; \mbox{\tiny
admissible}}\;\|v\|_{L^q(I; {\mathrm B}^{1}_{r,2}(\R^2))}\,.
$$
By interpolation argument, this Strichartz norm is equivalent to
$$
\|v\|_{L^\infty(I; H^1(\R^2))}+\|v\|_{L^4(I; {\mathrm B}^{1}_{8/3,2}(\R^2))}\,.
$$
As $ {\mathrm B}^{1}_{r,2}(\R^2)\hookrightarrow L^p(\R^2)$ for all
$r\leq p<\infty$ (and $r\leq p\leq\infty$ if $r>2$), it follows that
\bq \label{LqLp} \|v\|_{L^q(I; L^p)}\lesssim\|v\|_{\mbox{\tiny
ST}(I)},\quad \frac{1}{q}+\frac{2}{p}\leq 1\,. \eq Proposition
\ref{stri} will often be combined with the following elementary
bootstrap lemma.
\begin{lem}
\label{boot} Let $X(t)$ be a nonnegative continuous function on
$[0,T]$ such that, for every $0\leq t\leq T$, \bq \label{boot1}
X(t)\leq a+b\,X(t)^{\theta}, \eq where $a,b>0$ and $\theta>1$ are
constants such that \bq \label{boot2}
a<\left(1-\frac{1}{\theta}\right)\frac{1}{(\theta
b)^{1/{(\theta-1)}}},\quad X(0)\leq \frac{1}{(\theta
b)^{1/{(\theta-1)}}}\,. \eq Then, for every $0\leq t\leq T$, we
have \bq \label{boot3} X(t)\leq \frac{\theta}{\theta-1}\,a\,. \eq
\end{lem}
\begin{proof}[Proof of Lemma \ref{boot}] We sketch the proof for the convenience of the reader. The function $f: x\longmapsto b x^\theta-x+a$ is decreasing on $[0, (\theta b)^{1/{(1-\theta)}}]$ and increasing on $[(\theta b)^{1/{(1-\theta)}}, \infty[$. The first assumption in \eqref{boot2} implies that $f\left((\theta b)^{1/{(1-\theta)}}\right)<0$. As $f(X(t))\geq 0$, $f(0)>0$ and $X(0)\leq  (\theta b)^{1/{(1-\theta)}}$, we deduce by continuity \eqref{boot3}.
\end{proof}
\subsubsection{Logarithmic Inequalities}
It is well known that  the space $H^1({\mathbb R}^2)$ is not
included in  $L^\I({\mathbb R}^2)$. However, resorting to an
interpolation argument, we can estimate the $L^\I$ norm of
functions in $H^1({\mathbb R}^2)$, using a stronger norm but with
a weaker growth (namely logarithmic). More precisely, we have the
following logarithmic estimate which  also holds in any bounded
domain.
\begin{lem}[Logarithmic inequality \cite{DlogSob}, Theorem 1.3]
Let $0<\alpha<1$. For any real number
$\lambda>\frac{1}{2\pi\alpha}$, a constant $C_\lambda$ exists such
that for any function $\varphi $  belonging to $
H^1_0(|x|<1)\cap{\dot{\mathcal C}}^\alpha(|x|<1)$, we have
\begin{eqnarray}
\label{LS} \|\varphi\|_{L^\infty}^2\leq
\lambda\|\nabla\varphi\|_{L^2}^2 \log\left(C_\lambda+\frac{
\|\varphi\|_{{ \dot{\mathcal C}}^\alpha }}{
\|\nabla\varphi\|_{L^2} }\right),
\end{eqnarray}
where ${ \dot{\mathcal C}}^\alpha$ denotes the homogeneous
H\"older space of regularity index $\alpha$.
\end{lem}
We shall also need the following  version of the above inequality
which is available in the whole space.
\begin{lem}[\cite{DlogSob}, Theorem 1.3]
\label{Hmu}
 Let $0<\alpha<1$.  For any $\lambda>\frac{1}{2\pi\alpha}$ and
any $0<\mu\leq1$, a constant $C_{\lambda}>0$ exists such that for
any function $u\in H^1(\R^2)\cap{\mathcal C}^\alpha(\R^2)$, we
have
\begin{equation}
\label{H-mu} \|u\|^2_{L^\infty}\leq
\lambda\|u\|_{H_\mu}^2\log\left(C_{\lambda} +
\frac{8^\alpha\mu^{-\alpha}\|u\|_{{\mathcal
C}^{\alpha}}}{\|u\|_{H_\mu}}\,\,\,\right),
\end{equation}
where ${ {\mathcal C}}^\alpha$ denotes the inhomogeneous H\"older
space of regularity index $\alpha$ and  $H_\mu$  the Sobolev space
endowed with the norm $\|u\|_{H_\mu}^2:=\|\nabla
u\|_{L^2}^2+\mu^2\|u\|_{L^2}^2$.
\end{lem}
\subsubsection{Convergence in measure}
Similarly to  higher dimensions (see \cite{Ge1}), the concept of
convergence in measure occurs in the process of the proof of
Theorems \ref{Subcrit} and \ref{Crit}. For the convenience of the
reader, let us give an outline of this notion. In many cases, the
convergence in Lebesgue space $L^1$ is reduced to the convergence in
measure.
\begin{defi}
\label{CVmeas} Let $\Omega$ be a measurable subset of $\R^m$ and
$(u_n)$ be a sequence of measurable functions on $\Omega$. We say
that the sequence $(u_n)$ converges in measure to $u$ if, for every
$\varepsilon>0$,
$$
\Big|\{\, y\in\Omega\;;\;\;\;|u_n(y)-u(y)|\geq \varepsilon\,\}\Big|\to
0\quad\mbox{as}\quad n\to\infty,
$$
where $|B|$ stands for the Lebesgue measure of a measurable set
$B\subset\R^m$.
\end{defi}
It is clear that the convergence in $L^1$ implies the convergence in
measure. The contrary is also true if we require the boundedness in
some Lebesgue space $L^q$ with $q>1$. More precisely, we have the following well known result.
\begin{prop}
\label{CVMeas} Let $\Omega$ be a measurable subset of $\R^m$ with
finite measure and $(u_n)$ be a bounded sequence in $L^q(\Omega)$
for some $q>1$. Then, the sequence $(u_n)$ converges to $u$ in
$L^1(\Omega)$ if, and only if, it converges to $u$ in measure in
$\Omega$.
\end{prop}
\begin{proof}
The fact that the convergence in $L^1$ implies the convergence in
measure follows immediately from the following Tchebychev's inequality
$$
\varepsilon\;\Big|\{\, y\in\Omega\;;\;\;\;|u_n(y)-u(y)|\geq \varepsilon\,\}\Big|\;\leq\; \|u_n-u\|_{L^1}\;.
$$
To prove the converse, let us show first that $u$ belongs to $L^q(\Omega)$. Since the sequence $(u_n)$ converges to $u$ in measure, we get thanks to Egorov's lemma, up to subsequence extraction
$$
u_n\to u\quad\mbox{a.e.}\quad \mbox{in}\quad \Omega\;.
$$
The Fatou's lemma and the boundedness of $(u_n)$ in $L^q$ imply then
$$
\int_{\Omega}\,|u(y)|^q\,dy\leq \liminf_{n\to\infty}\,\int_{\Omega}\,|u_n(y)|^q\,dy\leq C\,.
$$
According to H\"older inequality, we have  for any fixed $\varepsilon>0$
\begin{eqnarray*}
\|u_n-u\|_{L^1}&=&\int_{\{|u_n-u|<\varepsilon\}}\,|u_n-u| +\int_{\{|u_n-u|\geq \varepsilon\}}\,|u_n-u|\\
&\leq& \varepsilon |\Omega|+\|u_n-u\|_{L^q}\,|\{|u_n-u|\geq \varepsilon\}|^{1-\frac{1}{q}},
\end{eqnarray*}
which ensures the result.
\end{proof}

\subsection{Subcritical case}

The aim of this section is to prove  that the nonlinear term does
not affect the behavior of the solutions in the  subcritical case.
By hypothesis  in that case, there exists some nonnegative real
$\r$ such that $\Limsup_{n\rightarrow\infty}\;E^n=1-\r$.
The main point for the proof of  Theorem \ref{Subcrit} is based on
the following lemma.
\begin{lem}
\label{subc-est} For every $T>0$ and $E_0<1$, there exists a
constant $C(T,E_0)$, such that every solution $u$ of the nonlinear
Klein-Gordon  equation (\ref{NLKG}) of energy $E(u)\leq E_0$,
satisfies \beq \label{subc-stri} \|u\|_{L^4([0,T]; {\cC}^{1/4})}\leq
C(T,E_0).
 \eeq
\end{lem}
\begin{proof}[Proof of Lemma \ref{subc-est}]
By virtue of  Strichartz estimate (\ref{e6}), we
have $$ \|u\|_{L^4([0,t]; {\cC}^{1/4})}\lesssim
E_0^{1/2}+\|f(u)\|_{L^1([0,t]; L^2(\R^2))}\,.$$ To estimate $f(u)$ in
$L^1([0,t]; L^2(\R^2))$, let us apply H\"older inequality $$ \|f(u)\|_{L^2}
\leq \|u\|_{L^{2+2/\varepsilon}}\;\|{\rm e}^{4\pi
u^2}-1\|_{L^{2(1+\varepsilon)}},$$ where $\varepsilon>0$ is chosen
small enough. This leads in view of Lemma \ref{apa1} to $$
\|f(u)\|_{L^2} \lesssim \|u\|_{H^1}\;{\rm
e}^{2\pi\|u\|_{L^\infty}^2}\;\|{\rm
e}^{4\pi(1+\varepsilon)u^2}-1\|_{L^1}^{\frac{1}{2(1+\varepsilon)}}.
$$
 The logarithmic inequality (\ref{H-mu}) yields for any fixed
 $\lambda>\frac{2}{\pi}$,
$$ {\rm e}^{2\pi\|u\|_{L^\infty}^2}\lesssim
\Big(C+\frac{\|u\|_{{\cC}^{1/4}}}{E_0^{1/2}}\Big)^{2\pi\lambda
E_0} $$ and  Trudinger-Moser inequality implies that  for
$\varepsilon>0$ small enough $$ \|{\rm
e}^{4\pi(1+\varepsilon)u^2}-1\|_{L^1}\leq \kappa. $$ Plugging
these estimates together, we obtain $$
\|u\|_{L^4([0,t] ;{\cC}^{1/4})}\lesssim E_0^{1/2}\Big(1+\Int_0^t
\Big(C+\frac{\|u\|_{{\cC}^{1/4}}}{E_0^{1/2}}\Big)^{\theta}d\tau\Big)
$$ where $\theta:=2\pi\lambda E_0$. Since $E_0<1$, we
can choose $\lambda>\frac{2}{\pi}$ such that
$\theta<4$. Using H\"older inequality in time, we deduce that
\begin{eqnarray*}
\|u\|_{L^4([0,t]; {\cC}^{1/4})}&\lesssim&
E_0^{1/2}\Big(1+t^{1-\theta/4}\left(t^{1/4}+E_0^{-1/2}
\|u\|_{L^4([0,t]; {\cC}^{1/4})}\right)^{\theta}\Big)\\
&\lesssim&E_0^{1/2}+T+E_0^{\frac{1-\theta}{2}}\,t^{1-\theta/4}\;\|u\|_{L^4([0,t]; {\cC}^{1/4})}^\theta\,.
\end{eqnarray*}
In the case where $\theta>1$, we set
$$t_{max}:=\left(\frac{C E_0^{\frac 12}}{E_0^{1/2}+T}\right)^{\frac{4(\theta-1)}{4-\theta}},$$
where $C$ is some constant. Then we obtain the desired result on the
interval $[0, t_{max}]$ by absorption argument (see Lemma
\ref{boot}). Finally, to get the general case we decompose
$[0,T]=\cup_{i=0}^{i=n-1}\,[t_i,t_{i+1}]$ such that
$t_{i+1}-t_i\leq t_{max}$. Applying the Strichartz estimate on
$[t_i,t]$ with $t\leq t_{i+1}$ and using the conservation of the
energy, we deduce $$ \|u\|_{L^4([t_i,t_{i+1}],  {\cC}^{1/4})}\leq
C(T, E_0)\,, $$ which yields the desired inequality. In the case where $\theta\leq 1$ we use a convexity argument and proceed exactly as above.

Notice that
similar argument was used in higher dimension (see \cite{IM}).
\end{proof}
\begin{rem}
\label{crit-rem}
Let us emphasize that in the critical case ($E_0=1$), the conclusion of Lemma \ref{subc-est} holds provided the additional assumption
\bq
\label{criticalh}
\|u\|_{L^\I([0,T]; {\cL})}<\frac{1}{\sqrt{4\pi}}\,.
\eq
The key point consists to estimate differently the term $\|{\rm e}^{4\pi u^2}-1\|_{L^{2(1+\varepsilon)}}$. More precisely, taking advantage of \eqref{criticalh} we write
\begin{eqnarray*}
\|{\rm e}^{4\pi
u^2}-1\|_{L^{2(1+\varepsilon)}}&\leq&\;{\rm
e}^{\frac{2\pi}{1+\varepsilon}\|u\|_{L^\infty}^2}\;\|{\rm
e}^{4\pi(1+2\varepsilon)u^2}-1\|_{L^1}^{\frac{1}{2(1+\varepsilon)}}\\&\leq& \kappa^{\frac{1}{2(1+\varepsilon)}}\;{\rm
e}^{\frac{2\pi}{1+\varepsilon}\|u\|_{L^\infty}^2},
\end{eqnarray*}
which leads to the result along the same lines as above.
\end{rem}

Let us now go to the proof of Theorem \ref{Subcrit}. Denoting
by~$w_n= u_n- v_n$, we can easily verify that~$w_n$ is the
solution of the nonlinear wave equation $$  \square w_n + w_n =
-f(u_n) $$ with null Cauchy data.\\

Under energy estimate, we obtain $$\|w_n\|_T \lesssim
\|f(u_n)\|_{L^1([0,T],L^2(\R^2))},$$ where~$\|w_n\|^2_T \eqdefa
\sup_{t\in [0,T]} E_c(w_n,t)$. Therefore, to prove that the
sequence $(u_n)$ is linearizable on $[0,T]$, it suffices to
establish that $$ \|f(u_n)\|_{L^1([0,T],L^2(\R^2))}\longrightarrow
0\quad\mbox{as}\quad n\rightarrow\infty. $$ Thanks to finite
propagation speed,  for any time~$t \in [0,T]$, the sequence
$f(u_n (t,\cdot))$ is uniformly supported in a compact subset~$K$
of~$ \R^2$. So, to prove that the sequence ~$(f(u_n))$ converges
strongly to~$0$ in~$L^1([0,T],L^2(\R^2))$, we shall follow the
strategy of P.~G\'erard in \cite{Ge1} which is firstly to
demonstrate that this sequence is bounded
in~$L^{1+\epsilon}([0,T],L^{2+\epsilon}(\R^2))$, for some
nonnegative ~$\epsilon $, and secondly to prove that it converges
to~$0$ in measure in~$[0,T]\times \R^2$.\\

Let us then begin by
estimate~$\|f(u_n)\|_{L^{1+\epsilon}([0,T],L^{2+\epsilon}(\R^2))}$,
for~$\epsilon $ small enough. Since, by definition we
have~$f(u_n)= -u_n\;({\rm e}^{4\pi u_n^2}-1)$, straightforward
computations imply that
 $$ \|f(u_n)\|^{2+\epsilon}_{L^{2+\epsilon}(\R^2)} \leq C {\rm
e}^{4\pi(1+\epsilon)\|u_n\|_{L^{\infty}}^2} \int_{\R^2}
|u_n|^{2+\epsilon}({\rm e}^{4\pi u_n^2}-1)dx.
 $$
 In other respects, using the obvious estimate
 $$\sup_{x\geq 0}(x^m {\rm e}^{-\gamma x^2})=\bigl(\frac{m}{2\gamma}\bigr)^{\frac m 2}{\rm e}^{-\frac m 2},$$
 we get, for any positive real~$\eta$
$$ \int_{\R^2} |u_n|^{2+\epsilon}({\rm e}^{4\pi u_n^2}-1)dx \leq
C_\eta \int_{\R^2} ({\rm e}^{(4\pi+\eta) u_n^2}-1)dx. $$ In
conclusion $$ \|f(u_n)\|^{2+\epsilon}_{L^{2+\epsilon}(\R^2)} \leq
C_\eta {\rm e}^{4\pi(1+\epsilon)\|u_n\|_{L^{\infty}}^2}
\int_{\R^2} \bigl({\rm e}^{(4\pi+\eta)(1-\r)^2 \bigl(\frac{
u_n}{1-\r}\bigr)^2}-1\bigr)dx.
 $$
Thanks to Trudinger-Moser estimate (\ref{Mos1}), we obtain
for~$\eta$ small enough \begin{eqnarray*}
\|f(u_n)\|^{2+\epsilon}_{L^{2+\epsilon}(\R^2)} &\leq C_\eta& {\rm
e}^{4\pi(1+\epsilon)\|u_n\|_{L^{\infty}}^2} \|u_n\|_{L^{2}}^2
\\&\leq C_\eta& {\rm e}^{4\pi(1+\epsilon)\|u_n\|_{L^{\infty}}^2}
E^n \\&\leq C_\eta& {\rm
e}^{4\pi(1+\epsilon)\|u_n\|_{L^{\infty}}^2},
\end{eqnarray*}
by energy estimate, using the fact that
$\Limsup_{n\rightarrow\infty}\;E^n<1-\r$.\\

Now, taking advantage of the logarithmic estimate (\ref{H-mu}), we
get for any  $\lambda>\frac{2}{\pi}$ and any $0<\mu\leq1$ $$ {\rm
e}^{4\pi(1+\epsilon)\|u_n\|_{L^{\infty}}^2} \lesssim
\left(C_\lambda+\frac{ \|u_n\|_{ {{\mathcal C}}^{\frac 1 4} }}{
\sqrt{(1-\r)(1+\mu^2)}
}\right)^{4\lambda\pi(1+\epsilon)(1-\r)(1+\mu^2)}. $$ We
deduce that
$$\|f(u_n)\|^{1+\epsilon}_{L^{1+\epsilon}([0,T],L^{2+\epsilon}(\R^2))}
\leq C_{\eta,\r}\int^T_0 \left(C_\lambda+ \|u_n\|_{{ {\mathcal
C}}^{\frac 1 4}
}\right)^{\frac{4\lambda\pi(1+\epsilon)^2(1-\r)(1+\mu^2)}{2+\epsilon}}dt.
$$ Choosing~$\lambda $ close to $\frac{2}{\pi}$, ~$\epsilon$
and~$\mu$ small enough such that~$\theta \eqdefa
\frac{4\lambda\pi(1+\epsilon)^2(1-\r)(1+\mu^2)}{2+\epsilon}<
4$, it comes by virtue of H\"older inequality
\bq
\label{L1L2}
\|f(u_n)\|^{1+\epsilon}_{L^{1+\epsilon}([0,T],L^{2+\epsilon}(\R^2))}
\leq C(\eta,\r,T)( T^\frac 1 4 + \|u_n\|_{L^4
([0,T]),{\cC}^{1/4})})^\theta.
\eq

 Lemma \ref{subc-est} allows to
end the proof of the first step, namely that in the subcritical
case the sequence~$(f(u_n))$ is bounded
in~$L^{1+\epsilon}([0,T],L^{2+\epsilon}(\R^2))$, for $\epsilon$
small enough.\\

Since~$\epsilon > 0$, we are then reduced as it is mentioned above
to prove that the sequence~$(f(u_n))$ converges to~$0$ in measure
in~$[0,T]\times \R^2$. Thus, by definition  we have  to prove that
for every~$\epsilon
> 0$, $$\Big|\bigl\{(t,x) \in [0,T]\times \R^2 , \quad |f(u_n)|
\geq \epsilon \bigr\}\Big| \longrightarrow 0\quad\mbox{as}\quad
n\rightarrow\infty\,.$$
The function~$f$ being continuous  at the
origin with $f(0)=0$, it suffices then to show that the sequence~$(u_n)$ converges
to~$0$ in measure.

Using the fact that~$(u_n)$ is supported in a fixed compact
subset of~$[0,T] \times \R^2$, we are led thanks to Rellich's
theorem and Tchebychev's inequality to prove that the
sequence~$(u_n)$ converges weakly to~$0$ in~$H^1([0,T] \times
\R^2)$. Indeed, assume that the sequence $(u_n)$ converges weakly to $0$ in~$H^1([0,T] \times
\R^2)$, then by Rellich's theorem $(u_n)$ converges strongly to $0$ in $L^2([0,T] \times\R^2)$. The Tchebychev's inequality
\bq
\label{tchebi}
\epsilon^2\;\Big|\bigl\{(t,x) \in [0,T]\times \R^2 , \quad |u_n(t,x)|
\geq \epsilon \bigr\}\Big|\leq \|u_n\|_{L^2}^2
\eq
 implies the desired result.

Let~$u$ be a weak limit of a subsequence~$(u_n)$. By virtue of Rellich's theorem and Tchebychev's inequality \eqref{tchebi}, the sequence $(u_n)$ converges to $u$ in measure. This leads to the convergence in measure of the sequence $f(u_n)$ to $f(u)$ under the continuity of the function $f$. Combining this information with the fact that ~$(f(u_{n}))$ is bounded in some~$L^q$ with~$q >
1$ and is uniformly compactly supported, we infer by Proposition \ref{CVmeas} that the convergence is also
distributional and~$u$ is a solution of the nonlinear Klein-Gordon
equation (\ref{NLKG}). Taking advantage of Lemma \ref{subc-est}, the compactness of the support and Estimate~\eqref{L1L2}, we deduce that~$f(u)\in
L^{1}([0,T],L^{2}(\R^2))$. This allows to apply energy method, and
shows that the energy of~$u$ at time~$t$ equals the energy of the
Cauchy data at~$t=0$, which is~$0$. Hence~$u\equiv 0$ and the
proof is complete.


\subsection{Critical case}

Our purpose here is  to prove Theorem \ref{Crit}. Let $T>0$ and
assume that \bq \label{crit-assum}
L:=\limsup_{n\to\infty}\;\|v_n\|_{L^\infty([0,T];{\cL})}<\frac{1}{\sqrt{4\pi}}.
\eq As it is mentioned above, ~$w_n= u_n- v_n$,  is the solution of
the nonlinear wave equation $$ \square w_n + w_n = -f(u_n) $$ with
null Cauchy data.\\ Under energy estimate, we have
$$\|w_n\|_T  \leq C \|f(u_n)\|_{L^1([0,T],L^2(\R^2))},$$
where~$\|w_n\|_T^2 \eqdefa \sup_{t\in [0,T]} E_c(w_n,t)$. It suffices then to prove that $$\|f(u_n)\|_{L^1([0,T],L^2(\R^2))}\to 0, \quad \mbox{as}\quad n \to \infty.$$ The idea
here is to split $f(u_n)$ as follows applying Taylor's formula
$$
f(u_n)=f(v_n+w_n)=f(v_n)+f'(v_n)\,w_n+\frac{1}{2}\;f''(v_n+\theta_n\,w_n)\,w_n^2,
$$
for some $0\leq \theta_n\leq 1$. The Strichartz inequality \eqref{e6} yields
(with $I=[0,T]$) \beq \nonumber
\|w_n\|_{\mbox{\tiny ST}(I)}&\lesssim&\|f(v_n)\|_{L^1([0,T]; L^2(\R^2))}+\|f'(v_n)\,w_n\|_{L^1([0,T]; L^2(\R^2))}\\&+&\|f''(v_n+\theta_n\,w_n)\,w_n^2\|_{L^1([0,T]; L^2(\R^2))}\\ \nonumber
\label{critestim} &\lesssim& I_n+J_n+K_n\,. \eeq
The term $I_n$ is the easiest term to treat. Indeed, by Assumption \eqref{crit-assum} we have
\bq
\label{critborn}
\|v_n\|_{L^\infty([0,T]; {\cL})}\leq\,\frac{1}{\sqrt{4\pi(1+\varepsilon)}},
\eq
for some $\varepsilon>0$ and $n$ large enough. This leads by similar arguments to the ones used in the proof of the subcritical case
$$
\|f(v_n)\|^{2+\eta}_{L^{2+\eta}(\R^2)} \leq C {\rm
e}^{4\pi(1-\eta)\|v_n\|_{L^{\infty}}^2} \int_{\R^2}
({\rm e}^{4\pi(1+3\eta) v_n^2}-1)dx.
$$
In view of \eqref{critborn} and the Logarithmic inequality, we obtain for $0<\eta<\frac{\varepsilon}{4}$ and $n$ large enough
$$
\|f(v_n)\|^{1+\eta}_{L^{1+\eta}([0,T],L^{2+\eta}(\R^2))}
\leq C(\eta,T)( T^\frac 1 4 + \|v_n\|_{L^4
([0,T]),{\cC}^{1/4})})^\theta,
$$
with $\theta=\frac{4\pi\lambda(1-\eta^2)}{2+\eta}$ and $0<\lambda-\frac{2}{\pi}\ll 1$. It follows by Strichartz estimate that
$(f(v_n))$ is bounded in $L^{1+\eta}([0,T]; L^{2+\eta}(\R^2))$.

Since $v_n$ solves the linear Klein-Gordon equation with Cauchy data weakly convergent to $0$ in $H^1\times L^2$, we deduce that $(v_n)$ converges weakly to $0$ in $H^1([0,T]\times\R^2)$. This implies that $f(v_n)$
converges to $0$ in measure. This finally leads, using Proposition
\ref{CVMeas}, the fixed support property and interpolation argument, to the convergence of the sequence
$(f(v_n))$ to $0$ in $L^1([0,T]; L^2(\R^2))$.

Concerning the second term $J_n$, we will show that
\bq \label{claim2} J_n\leq \,\varepsilon_n\,\|w_n\|_{\mbox{\tiny
ST}(I)}, \eq
where $\varepsilon_n\to 0$.\\
Using H\"older inequality, we infer that
$$
J_n=\|f'(v_n)\,w_n\|_{L^1([0,T]; L^2(\R^2))}\leq
\|w_n\|_{L^{1+\frac{1}{\eta}}([0,T]; L^{2+\frac{2}{\eta}}(\R^2))}\,\|f'(v_n)\|_{L^{1+\eta}([0,T]; L^{2+2\eta}(\R^2))}\,
$$
Arguing exactly in the same manner as for $I_n$, we prove that for $\eta \leq \eta_0 $ small enough the sequence
$(f'(v_n))$ is bounded in $L^{1+\eta}([0,T]; L^{2(1+\eta)}(\R^2))$ and converges to $0$ in measure which ensures its convergence to $0$ in $L^1([0,T]; L^2(\R^2))$. Hence the sequence $(f'(v_n))$ converges to $0$ in $L^{1+\eta}([0,T]; L^{2+2\eta}(\R^2))$, for $\eta < \eta_0 $,  by interpolation argument.
This completes the proof of \eqref{claim2} under the Strichartz estimate \eqref{LqLp}. \\
For the last (more difficult) term we will establish that
\bq \label{claim3} K_n\leq \,\varepsilon_n\,\|w_n\|_{\mbox{\tiny ST}(I)}^2,\quad \varepsilon_n\to 0,
\eq provided that \bq \label{claim31}
\limsup_{n\to\infty}\,\|w_n\|_{L^\infty([0,T]; H^1)}\leq
\frac{1-L\,\sqrt{4\pi}}{2}. \eq
By H\"older inequality, Strichartz estimate and convexity argument, we infer that
\begin{eqnarray*}
K_n&\leq& \|w_n^2\|_{L^{1+\frac{1}{\eta}}([0,T]; L^{2+\frac{2}{\eta}}(\R^2))}\,\|f''(v_n+\theta_n\,w_n)\|_{L^{1+\eta}([0,T]; L^{2+2\eta}(\R^2))}\\
&\leq& \|w_n\|_{\mbox{\tiny ST}(I)}^2\left(\|f''(v_n)\|_{L^{1+\eta}([0,T]; L^{2+2\eta}(\R^2))}+\|f''(u_n)\|_{L^{1+\eta}([0,T]; L^{2+2\eta}(\R^2))}\right)\;.
\end{eqnarray*}
According to the previous step, we are then led to prove that for $\eta$ small enough
\bq
\label{critsecond}
\|f''(u_n)\|_{L^{1+\eta}([0,T]; L^{2+2\eta}(\R^2))}\to 0\;.
\eq
Arguing exactly as in the subcritical case, il suffices to establish that the sequence $(f''(u_n))$ is bounded in ${L^{1+\eta_0}([0,T]; L^{2+2\eta_0}(\R^2))}$ for some $\eta_0>0$. Let us first point out that the assumption \eqref{claim31} implies that
\begin{eqnarray*}
\limsup_{n\to\infty}\,\|u_n\|_{L^\infty([0,T]; {\cL})}&\leq&\limsup_{n\to\infty}\,\|v_n\|_{L^\infty([0,T]; {\cL})}+
\limsup_{n\to\infty}\,\|w_n\|_{L^\infty([0,T]; {\cL})}\\&\leq&
L+\frac{1}{\sqrt{4\pi}}\,\|w_n\|_{L^\infty([0,T]; H^1)}\\&\leq&\frac{1}{2}\left(L+\frac{1}{\sqrt{4\pi}}\right)<\frac{1}{\sqrt{4\pi}}\,.
\end{eqnarray*}
This ensures thanks to Remark \ref{crit-rem} the boundedness of the sequence $(u_n)$ in $L^4([0,T], {\cC}^{1/4})$ which leads to \eqref{critsecond} in a similar way than above.
Now we are in position to end of the proof of Theorem \ref{Crit}.
According to \eqref{claim2}- \eqref{claim3}, we can
rewrite \eqref{critestim} as follows \bq \label{final-esti} {\mathbf
X}_n(T)\lesssim I_n+\varepsilon_n\,{\mathbf X}_n(T)^2, \eq where ${\mathbf
X}_n(T):=\|w_n\|_{\mbox{\tiny ST}([0, T])}$. In view of Lemma \ref{boot}, we deduce that
$$
{\mathbf X}_n(T)\lesssim \varepsilon_n\,.
$$
This leads to the desired result under \eqref{claim31}. To remove the assumption \eqref{claim31}, we use classical arguments. More precisely, let us set
\bq
\label{closing}
T^*:=\sup\left\{0\leq t\leq T;\quad \limsup_{n\to\infty}\,\|w_n\|_{L^\infty([0,t]; H^1)}\leq \nu\right\},
\eq
where $\nu:=\frac{1-L\,\sqrt{4\pi}}{2}$. Since $w_n(0)=0$, we have $T^*>0$. Assume that $T^*<T$ and apply the same arguments as above, we deduce that ${\mathbf X}_n(T^*)\to 0$. By continuity this implies that $\limsup_{n\to\infty}\,\|w_n\|_{L^\infty([0,T^*+\epsilon]; H^1)}\leq \nu$ for some $\epsilon$ small enough. Obviously, this contradicts the definition of $T^*$ and hence $T^*=T$.

\section{Appendix}


\subsection{Appendix A: Some known results on Sobolev embedding}
\label{apA}


\begin{lem}
\label{apa1} $H^1(\R^2)$ is embedded into $L^p(\R^2)$ for all $2\leq
p<\infty$ but not in $L^\infty(\R^2)$.\end{lem}
\begin{proof}[Proof of Lemma \ref{apa1}]
Using  Littlewood-Paley decomposition and  Bernstein inequalities
(see for instance \cite{DP}), we infer that
\begin{eqnarray*}
\|v\|_{L^p}&\leq&\sum_{j\geq -1}\,\|\triangle_j v\|_{L^p},\\
&\leq&C\sum_{j\geq -1}\,2^{-\frac{2j}{p}}\,2^j\,\|\triangle_j
v\|_{L^2}.\end{eqnarray*} Taking advantage of Schwartz inequality,
we deduce that
$$ \|v\|_{L^p}\leq \,C \Big(\sum_{j\geq
-1}\,2^{-\frac{4j}{p}}\Big)^{\frac{1}{2}}\,\|v\|_{H^1}
\leq\,C_p\,\|v\|_{H^1},$$ which achieves the proof of the embedding
for $2\leq p<\infty$.   However, $H^1(\R^2)$  is not included
in~$L^{\infty}(\R^2 )$. To be convenience, it suffices to consider
 the function~$u$  defined
by
$$u(x)= \vf(x) \log(-\log |x|)$$ for some smooth function~$\vf$
supported in~$B(0,1)$ with value~$1$ near~$0$. \end{proof} It will
be useful to notice, that in the radial case, we have the following
estimate which implies the control of the $L^\infty$-norm far away
from the origin.
\begin{lem}
\label{apa3}
 Let $u\in H^1_{rad}(\R^2)$ and $1\leq p<\infty$. Then
$$
|u(x)|\leq\frac{C_p}{r^{\frac{2}{2+p}}}\,\|u\|_{L^p}^{\frac{p}{p+2}}\|\nabla
u\|_{L^2}^{\frac{2}{p+2}},
$$
with $ r= |x|$.  In particular \bq \label{Bound}
|u(x)|\leq\frac{C_2}{r^{\frac{1}{2}}}\,\|u\|_{L^2}^{\frac{1}{2}}\|\nabla
u\|_{L^2}^{\frac{1}{2}}\leq
\frac{C_2}{r^{\frac{1}{2}}}\,\|u\|_{H^1}. \eq
\end{lem}
\begin{proof}[Proof of Lemma \ref{apa3}] By density, it suffices to consider smooth compactly supported functions.
Let us then consider $u(x)=\varphi(r)$, with $\varphi\in {\mathcal
D}([0,\infty[)$. Obviously, we have
$$
\varphi(r)^{\frac{p}{2}+1}=-\frac{p+2}{2}\;\int_r^\infty\;\varphi'(s)\,\varphi^{\frac{p}{2}}(s)\;ds\,.
$$
Hence
\begin{eqnarray*}
|\varphi(r)|^{\frac{p}{2}+1}&\leq&\frac{p+2}{2r}\;\int_r^\infty\;|\varphi'(s)|\,|\varphi(s)|^{\frac{p}{2}}\,s\;ds,\\
&\leq&\frac{p+2}{2r}\;\|\nabla u\|_{L^2}\,\|u\|_{L^p}^{\frac{p}{2}}.
\end{eqnarray*}
This achieves the proof of the lemma.
\end{proof}
\begin{rem}
In the general case,  the embedding of $H^1 (\R^2)$ into $L^p
(\R^2)$ is not compact. This observation can be illustrated by the
following example: $u_n(x)=\varphi(x+x_n)$ with
$0\neq\varphi\in{\mathcal D}$ and $|x_n|\to\infty$. However, by
virtue  of Rellich-Kondrachov's theorem, this embedding is compact
in the case of $H_K^1 (\R^2)$ the subset of functions of $H^1
(\R^2)$ supported in the compact $K$. Moreover, in the radial case,
the following compactness result holds.
\end{rem}
\begin{lem}
\label{apa4} Let $2< p<\infty$. The embedding $H^1_{rad}(\R^2)$ in
$L^p(\R^2)$ is compact.
\end{lem}
\begin{proof}[Proof of Lemma \ref{apa4}] The proof is quite standard and can be found in many references (see for example \cite{BL, Kavian, Strauss}).
We sketch it here for the sake of completeness.  $(u_n)$ being  a
sequence in $H^1_{rad}(\R^2)$ which converges weakly to $u\in
H^1_{rad}(\R^2)$, let us set $v_n:=u_n-u$. The problem is then
reduced  to the proof of the fact  that $\|v_n\|_{L^p}$ tends to
zero. On the one hand, using the above Lemma, we get for any $R>0$,
$$
\int_{|x|>R}\;|v_n(x)|^p\,dx=\int_{|x|>R}\;|v_n(x)|^{p-2}\,|v_n(x)|^2\,dx\leq\,
C\,R^{-\frac{p-2}{2}}\,.
$$
On the other hand, we know by Rellich-Kondrachov's theorem that the
injection $H^1(|x|\leq R)$ into $L^p(|x|\leq R)$ is compact. This
ends the proof.\end{proof}

\begin{rem}
 $H^1_{rad}(\R^2)$ is not compactly embedded in $L^2(\R^2)$. To see this,
it suffices to  consider the family
 $u_n(x)=\frac{1}{\alpha_n}\,{\rm e}^{-|\frac{x}{\alpha_n}|^2}$ where $(\alpha_n)$ is a sequence of nonnegative real numbers tending to infinity.
 One can easily show that $(u_n)$ is bounded in $H^1$ but cannot have a subsequence converging strongly in $L^2$.\\
 \end{rem}


\subsection{Appendix B: Some additional properties on Orlicz spaces}
\label{apB}

Here we recall some well known properties of Orlicz spaces. For a
complete presentation and more details, we refer the reader to
\cite{Orlicz-Book}. The first result that we state here deals with
the connection between Orlicz spaces and Lebesgue spaces $L^1$ and
$L^\infty$.
\begin{prop} We have\\
{\bf a)} $\left(L^\phi, \|\cdot\|_{L^\phi}\right)$ is a Banach space.\\
{\bf b)} $L^1\cap L^\infty\subset L^\phi\subset L^1+L^\infty$.\\
{\bf c)} If $T : L^1\to L^1$ with norm $M_1$ and $T : L^\infty\to
L^\infty$ with norm $M_\infty$,  then $T : L^\phi\to L^\phi$ with
norm $\leq C(\phi) \Sup(M_1,M_\infty)$.
\end{prop}
The following result concerns the behavior of Orlicz norm against
convergence of sequences.
\begin{lem}
\label{monoto}
 We have the following properties\\
{\bf a)} {\it Lower semi-continuity}:
$$
u_n\to u\quad\mbox{\sf a.e.}\quad\Longrightarrow\quad\|u\|_{\cL}\leq
\Liminf\|u_n\|_{\cL}.
$$
{\bf b)} {\it Monotonicity}:
$$
|u_1|\leq|u_2|\quad\mbox{\sf
a.e.}\quad\Longrightarrow\quad\|u_1\|_{\cL}\leq\|u_2\|_{\cL}.
$$
{\bf c)} {\it Strong Fatou property}:
$$
0\leq u_n\nearrow u\quad\mbox{\sf
a.e.}\quad\Longrightarrow\quad\|u_n\|_{\cL}\nearrow\|u\|_{\cL}.
$$
\end{lem}

Let us now stress that besides the topology induced by its norm, the
Orlicz space ${\cL}$ is equipped with one  other topology, namely
the mean topology.  More precisely,
\begin{defi}
A sequence $(u_n)$ in ${\cL}$ is said to be mean (or modular)
convergent to $u\in\cL$, if
$$
\Int\,\phi(u_n-u)\,dx\longrightarrow 0.
$$
It is said strongly (or norm) convergent to $u\in\cL$, if
$$
\|u_n-u\|_{\cL}\longrightarrow 0.
$$
\end{defi}
Clearly there is no equivalence between these convergence notions. Precisely, the strong convergence implies the modular convergence but the converse is false as shown by taking the Lions's functions $f_\alpha$. \\
To end this subsection, let us mention that our Orlicz space ${\cL}$ behaves like $L^2$ for functions in $H^1\cap L^\infty$.
\begin{prop}
\label{Orl-L2}
For every $\mu>0$ and every function $u$ in $H^1\cap L^\infty$, we have
\bq
\label{OrL2}
\frac{1}{\sqrt{\kappa}}\,\|u\|_{L^2}\leq \|u\|_{{\cL}}\leq \mu+\frac{{\rm e}^{\frac{\|u\|_{L^\infty}^2}{2\mu^2}}}{\sqrt{\kappa}}\,\|u\|_{L^2}\,.
\eq
\end{prop}
\begin{proof}
The left hand side of \eqref{OrL2} is obvious. The second inequality follows immediately from the following simple observation
$$
\left\{ \lambda\geq \mu+\frac{{\rm e}^{\frac{\|u\|_{L^\infty}^2}{2\mu^2}}}{\sqrt{\kappa}}\,\|u\|_{L^2}\right\}\subset \left\{\lambda>0\,;\quad\int\left({\rm e}^{\frac{|u(x)|^2}{\lambda^2}}-1\right)\,dx\leq \kappa\right\}\,.
$$
Indeed, assuming $\lambda\geq \mu+\frac{{\rm e}^{\frac{\|u\|_{L^\infty}^2}{2\mu^2}}}{\sqrt{\kappa}}$, we get
\begin{eqnarray*}
\int\left({\rm e}^{\frac{|u(x)|^2}{\lambda^2}}-1\right)\,dx&\leq&\int \,\frac{|u(x)|^2}{\lambda^2}\,{\rm e}^{\frac{|u(x)|^2}{\lambda^2}}\,dx\\
&\leq&\frac{{\rm e}^{\frac{\|u\|_{L^\infty}^2}{\mu^2}}}{\lambda^2}\,\|u\|_{L^2}^2\\
&\leq&\kappa.
\end{eqnarray*}
\end{proof}


\subsection{Appendix C: $BMO$ and ${\mathcal L}$}


Now, we shall discuss the connection between the Orlicz space
${\mathcal L}$ and {\rm BMO}. At first, let us recall the
following well known embeddings $$ H^1\hookrightarrow {\rm
BMO}\cap L^2,\quad L^\infty\hookrightarrow{\rm
BMO}\hookrightarrow{\rm B}^0_{\infty,\infty},\quad
H^1\hookrightarrow{\mathcal L}\hookrightarrow\bigcap_{2\leq
p<\infty} L^p. $$ However, there is no comparison between
${\mathcal L}$ and {\rm BMO} in the following sense.
\begin{prop}
\label{bmo} We have
$$
{\mathcal L}\not\hookrightarrow{\rm BMO}\cap L^2\quad\mbox{and}\quad
{\rm BMO}\cap L^2\not\hookrightarrow{\mathcal L}.
$$
\end{prop}
\begin{proof}[Proof of Proposition \ref{bmo}]

Let us consider $g_\alpha(r,\theta)=f_\alpha(r)\,{\rm e}^{i\theta}$
and $B_\alpha=B(0, {\rm e}^{-\frac{\alpha}{2}})$. Clearly we have
$$
\int_{B_\alpha}\;g_\alpha=0\,.
$$
Moreover
\begin{eqnarray*}
 \frac{1}{|B_\alpha|}\int_{B_\alpha}\;|g_\alpha|&=&2{\rm e}^{\alpha}\,\int_0^{{\rm e}^{-\alpha}}\,\sqrt{\frac{\alpha}{2\pi}}\,r\,dr+\int_{{\rm e}^{-\alpha}}^{{\rm e}^{-\frac{\alpha}{2}}}\,-\frac{\log r}{\sqrt{2\pi\alpha}}\,r\,dr\\&=& \frac{\sqrt{\alpha}}{2\sqrt{2\pi}}+\frac{1-{\rm e}^{-\alpha}}{2\sqrt{2\pi\alpha}}\,.
\end{eqnarray*}
Hence
$$
\|g_\alpha\|_{\rm BMO}\to\infty\quad\mbox{as}\quad
\alpha\to\infty\,.
$$
Since $\|g_\alpha\|_{\mathcal L}=\|f_\alpha\|_{\mathcal
L}\to\frac{1}{\sqrt{4\pi}}$, we deduce that
$$
{\mathcal L}\not\hookrightarrow{\rm BMO}\cap L^2\,.
$$
To show that  ${\rm BMO}\cap L^2$ is not embedded in ${\mathcal L}$,
we shall use the following {\bf sharp} inequality (see
\cite{Kozono})
\begin{equation}
\label{sharp-bmo} \|u\|_{L^q}\leq C\,q\|u\|_{\rm BMO\cap L^2},\quad
q\geq 2,
\end{equation}
together with the fact that (for $u\neq 0$),
\begin{equation}
\label{L} \int_{\R^2}\,\Big({\rm e}^{\frac{|u(x)|^2}{\|u\|_{\mathcal
L}^2}}-1\Big)\,dx\leq \kappa\,.
\end{equation}
Indeed, let us suppose that ${\rm BMO}\cap L^2$ is embedded in
${\mathcal L}$. Then, for any integer $q\geq 1$,
$$
\|u\|_{L^{2q}}\leq
\kappa^{1/2q}\,\left(q!\right)^{1/2q}\,\|u\|_{\mathcal L}\leq
C\,\kappa^{1/2q}\,\left(q!\right)^{1/2q}\,\|u\|_{\rm BMO\cap L^2}\,
$$
which contradicts \eqref{sharp-bmo} since
$$
\left(q!\right)^{1/2q}\sim {\rm e}^{-1/2}\,\sqrt{q},
$$
where $\sim$ is used to indicate that the ratio of the two sides
goes to $1$ as $q$ goes to $\infty$.
\end{proof}

\end{document}